\documentclass[a4paper]{amsart}
\usepackage{a4wide}
\usepackage{amssymb}
\usepackage[hidelinks]{hyperref}
\usepackage{graphicx}
\usepackage{tikz}
\usepackage{overpic}
\usepackage{breakcites}
\usepackage{enumitem}
\usepackage[export]{adjustbox}
\usepackage{lipsum}
\usepackage{pinlabel}
\usepackage{ulem}
\usepackage{lmodern}
\usepackage{microtype}

\newcommand{\R}{\mathbb{R}}
\newcommand{\C}{\mathbb{C}}

\newcommand{\re}{\operatorname{Re}}
\newcommand{\im}{\operatorname{Im}}

\theoremstyle{plain}
\newtheorem{theorem}{Theorem}[section]
\newtheorem{lemma}[theorem]{Lemma}

\newtheorem{proposition}[theorem]{Proposition}

\theoremstyle{definition}

\newtheorem{remark}[theorem]{Remark}

\def \cH {\mathcal{H}}

\begin{document}

\title[Gyrating H'-T surfaces]{The chiral gyrating H'-T surface family:\\ construction from the dual \texttt{qtz}--\texttt{qzd} nets and\\ existence proof using a toroidal Weierstrass method}

\author{Hao Chen}
\author{Shashank G.\ Markande}

\author{Matthias Saba}

\author{Gerd E. Schr\"oder-Turk}

\author{Elisabetta A.\ Matsumoto}
\address[Chen]{ShanghaiTech University, Institute of Mathematical Sciences, 201210  Pudong New District, Shanghai, China}
\email[Chen]{chenhao5@shanghaitech.edu.cn}
\address[Markande, Matsumoto]{School of Physics, Georgia Tech, Atlanta, GA 30318, USA}
\address[Saba]{Adolphe Merkle Institute, University of Fribourg, 1700, Switzerland}
\address[Schr\"oder-Turk]{Murdoch University, School of Mathematics, Statistics, Chemistry and Physics, Perth, Australia; The Australian National University; Research School of Physics, Department of Materials Science, Canberra, Australia}
\email[Schr\"oder-Turk]{g.schroeder-turk@murdoch.edu.au}


\begin{abstract}
	This paper provides a construction and existence proof for a 1-parameter family
	of chiral unbalanced triply-periodic minimal surfaces of genus 4. We name these
	{\textit{gyrating H'-T} surfaces, because they are related to Schoen's H'-T
		surfaces in a similar way as the Gyroid is to the Primitive surface. Their
		chirality is manifest in a screw symmetry of order six. The two labyrinthine
		domains on either side of the surface are not congruent, rather one
		representing the quartz net (\texttt{qtz}) and the other one the dual of the
		quartz net (\texttt{qzd}). The family tends to the Scherk saddle tower in one
		limit and to the doubly periodic Scherk surface in the other. The motivation
		for the construction was to construct a chiral tunable unbalanced surface
		family, originally as a template for photonic materials. The numeric
		construction is based on reverse-engineering of the tubular surface of two
	suitably chosen dual nets, using the \textit{Surface Evolver}} to minimize area
	or curvature variations.  The existence is proved using Weierstrass
	parametrizations defined on the branched torus. 
\end{abstract}

\maketitle

Triply-periodic minimal surfaces (TPMSs) are periodic symmetric-saddle surfaces
that divide space into two network-like domains.  They are epitomized by Alan
Schoen's \textit{Gyroid} surface \cite{Schoen_1970,Schoen_2012}. The
interdisciplinary interest in triply-periodic bicontinuous structures is driven
by the occurrence of related structures in chemistry and biology
\cite{Hyde_1997,hydeUbiquitousSRS} and as a useful material design such as the
now common 'gyroid infill' in 3d printing applications. The cubic (G)yroid and
its two close relatives, namely Schwarz' (D)iamond and (P)rimitive surfaces
are, among the many triply-periodic minimal surfaces, those that are most
widely found in nano- or micro-structured materials and tissues.

The study of TPMS has advanced through contributions from both the mathematical
and natural sciences. The productivity of this relationship, although perhaps
not always fully recognized \cite{SchoenObituary}, is epitomized by Schoen's
discovery of his Gyroid. Schoen, a trained physicist employed by NASA,
constructed the Gyroid using an 'infinite polyhedral model' and devised a
Weierstrass parametrization for the Gyroid \cite{Schoen_1970,Schoen_2012}),
with mathematicians Gro{\ss}e-Brauckmann and Wohlgemuth later proving its
existence and that it is a minimal embedding \cite{kgb1996}.
Gro{\ss}e-Brauckmann went on to establish the existence (and partial
uniqueness) of constant-mean-curvature companion surfaces to the gyroid
\cite{KGB_1997} as well as numerical constructions of these that have since
been used in physical applications
\cite{ShearmanSeddonCMC2007,SchrderTurk2013}.

\begin{figure*}[t!]
  	\centering
  	\includegraphics[width=\textwidth]{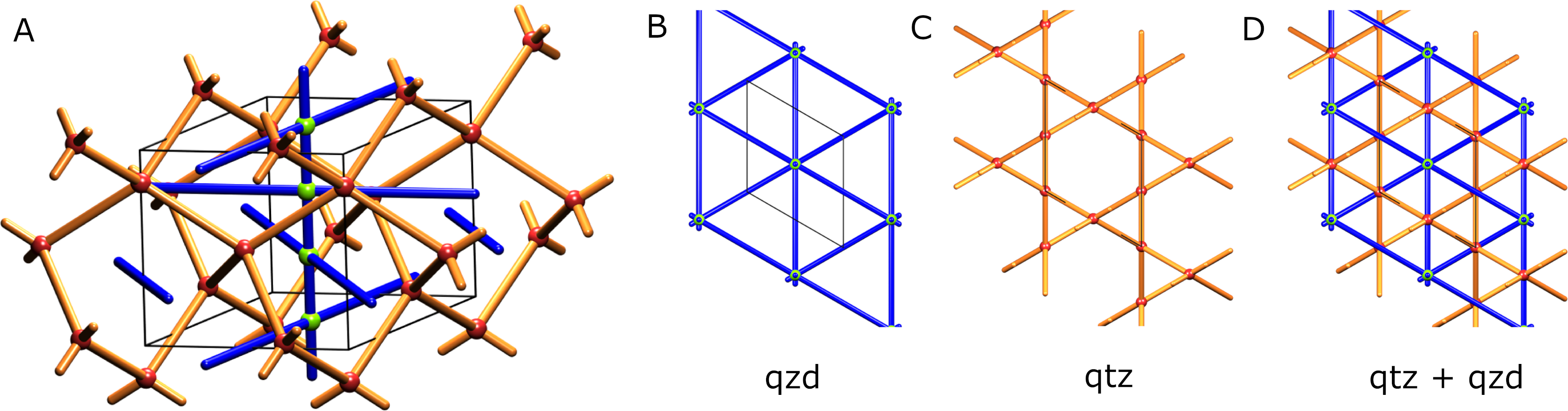}
		\caption{
			The pair of interpenetrating dual nets that lead to the
			\texttt{qtz}-\texttt{qzd} surface family, the quartz (\texttt{qtz}) net
			(orange) and its proper dual {\texttt{qzd}} network (blue). (A) Perspective
			view; (B-D) top views along the $c$ axis of the \texttt{qzd} (A),
			\texttt{qtz} (B) \texttt{qzd} and \texttt{qtz} nets. The translational unit
			cell ($P6_222$) is the gray frame.
		}
		\label{fig:qtz-qzd-nets}
\end{figure*}

A similar interdisciplinary interplay established our understanding of families
of TPMSs that contain the Gyroid, Diamond, or Primitive surfaces as particular
members, and therefore provide transition pathways between these surfaces
within the space of TPMSs. Meeks mathematically described a five-parameter
family containing the cubic P and D surfaces, but not the Gyroid
\cite{Meeks1975,meeks1990}. Hyde and Fogden enumerated Weierstrass
parametrisations of an extended set of these surfaces \cite{Fogden_i_1992,
Fogden_ii_1992, Fogden_iii_1993} including surface families rG and tGL
containing the Gyroid \cite{fogden1999}; these were later examined
mathematically by Weyhaupt \cite{weyhaupt2008} and their existence eventually
proved by Chen \cite{Chen_2021}. Numerical analysis of their packing and
curvature properties was carried out by Fogden \& Hyde \cite{fogden1999} and
others \cite{schroderturk2006, SchrderTurk2011,Mickel2012}. Recent work in
mathematics and materials science  has identified further deformation families,
particularly an alternative deformation family of the Diamond surface
\cite{Chen2021a} that is different from the classical tD family, an
orthorhombic deformation family of Schwarz' H surfaces \cite{Chen2021b}, and a
deformation of a (shifted) Gyroid \cite{Wang2024}; this deformation may relate
to the recent finding that the cubic Diamond is merely a saddle-point (rather
than an unstable minimum) of the free-energy functional for a copolymeric model
\cite{DimitriyevGrason2025}.

Over the years, research in both natural sciences and mathematics has proposed
numerous ways to describe, enumerate or parametrize new TPMSs.  In the natural
sciences, there are approaches based on space group or crystallographic
symmetries by Fischer $\&$ Koch (\cite{Koch_1988} and other papers in that
series) and Lord $\&$ Mackay \cite{Lord_2003}, on tilings of hyperbolic plane
by Sadoc $\&$ Charvolin \cite{Sadoc_1989},  on Weierstrass formula derived from
Schwarz triangular tilings of the two sphere by Fogden $\&$ Hyde
\cite{Fogden_i_1992, Fogden_ii_1992}, and other explicit parametrizations
\cite{Gandy_i_2000}.  In mathematics, there are approaches based on a
Schwarz--Christoffel formula for periodic polygons in the plane by Fujimori
$\&$ Weber \cite{Fujimori_2009}, on the concept of discrete minimal surfaces
\cite{Pinkall_1993,Karcher_1996}, on gluing methods~\cite{traizet2008,
chen2022, chen2024}. New TPMSs were recently constructed by deliberately
breaking symmetries \cite{Chen2021a, Chen2021b, Chen_2021}.

In addition to exact descriptions, methods to approximate or simulate TPMSs
have been instrumental in their exploration: Nodal surface formulae that
approximate minimal or CMC surfaces \cite{vonSchnering1991} have become popular
for their ease of use. Furthermore, a tool that has been widely used in
mathematics and natural sciences for the exploration of TPMSs is Kenneth
Brakke's \textit{Surface Evolver} \cite{Brakke_1992}. This software package
uses a conjugate gradient solver to minimize energy functionals (including the
area functional, Willmore functional, etc.) on a surface, and integrates
functionality for crystallographic symmetries, boundary conditions, volume
constraints, etc. Brakke's extensive TPMS library is testimony to its
usefulness \cite{BrakkeTPMSWebArchive} as are many other applications in the
context of area-minimizing surfaces \cite{KGB_1997,Hyde_2009} and its use in
this work. 

Some properties of a minimal surface relate to intrinsic properties, unrelated
to the embedding in $\mathbb{E}^3$. In particular, the Gauss curvature $K$, the
integral of which provides the Euler characteristic and the variations of which
are a homogeneity measure relevant for self-assembly \cite{Hyde_1990}, is an
intrinsic property. There are also many extrinsic properties. In particular,
the relationship between the minimal surface interface and the two skeletal
(medial) graphs or surfaces often used to illustrate the connectivity and
geometry of the two labyrinthine domains on either side of the surface. 

This paper is the result of another interplay between mathematics and natural
sciences on the topic of TPMSs: We here describe the numerical and analytic
constructions of a TPMS without in-surface symmetries (that is, without mirror
planes and without two-fold rotation axes embedded in the surface). The absence
of such symmetries means that a surface patch bounded by straight lines or a
plane line of curvature does not exist; therefore a bottom-up approach starting
from a Plateau patch is not possible. 

Our numerical construction of the gyrating H'-T is a top-down approach. We
start with two specific interpenetrating nets (namely the chiral
\texttt{qtz}--\texttt{qzd} pair), with specific topology and symmetry, with the
goal of determining a minimal surface interface that separates those nets. We
determine a mesh approximation using a numerical construction in Surface
Evolver (based on a procedure suggested by de Campo, Hyde et al
\cite{Hyde_2009}), which is good enough to determine the location and nature of
the flat points (those with zero Gauss curvature).

Based on that information, we determine a Weierstrass parametrization of the
surface. It turns out that this can be done in one of two ways. One way,
explored in an earlier unpublished version of this manuscript by a subset of
the current authors (available as an arXiV preprint \cite{preprint}), is using
the knowledge of the flat points to determine a Weierstrass function in the
spirit of work by Fogden \& Hyde \cite{Fogden_i_1992,Fogden_ii_1992},
particularly in relation to the so-called \textit{irregular} class
\cite{Fogden_iii_1993}. This is based on the ``traditional'' spherical
Weierstrass parametrization, which is the inverse of the stereographically
projected Gauss map.

Before joining the project, Chen had been working on the existence proof of the
tG and rGL families~\cite{Chen_2021} and a construction of minimal surfaces by
gluing Scherk saddle towers~\cite{chen2024}.  Upon reading~\cite{preprint}, he
noticed that the Scherk tower limit of the family can be produced by the gluing
method~\cite{chen2024}.  Moreover, he recognized that the gyrating H'-T
surfaces are to the H'-T surfaces what the tG surfaces are to the tP surfaces.
Therefore, the existence can be proved using a Weierstrass parametrization
defined on a torus, see section \ref{sec:Weierstrass-construction-hao}, in a
similar manner as was done for the deformation families of the Gyroid
\cite{Chen_2021}.

\section{Numerical construction from the dual interthreaded
\texttt{qtz}--\texttt{qzd} nets}

We construct a unique triply-periodic minimal surface from a dual pair of
interpenetrating nets (specifically the quartz net and its dual). A tubular
representation of one of the nets is 'evolved' numerically to a TPMS. One of
its two disjoint domains is homeomorphic to a tubular neighbourhood of the
\texttt{qtz} net while the other is homeomorphic to the \texttt{qzd} net. 

\subsection{Dual pair of \texttt{qtz} and \texttt{qzd} nets}

\begin{figure*}[b]
	\centering
	\includegraphics[width=\textwidth]{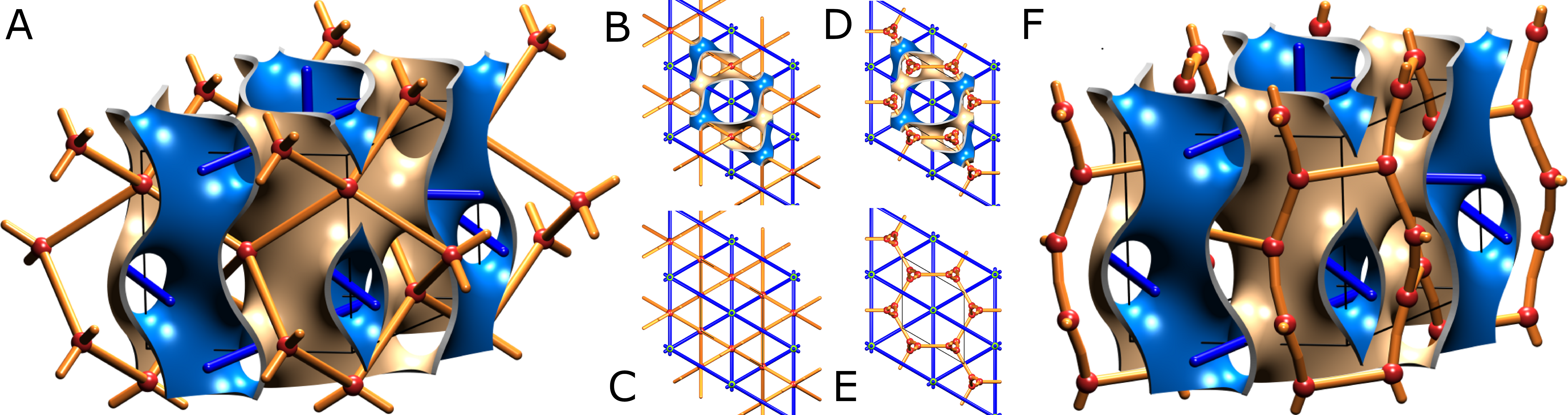}
	\caption{
		The numerically evolved minimal surface that separates the \texttt{qtz} and
		the \texttt{qzd} net. (A-C) the networks shown are the \texttt{qtz} and the
		\texttt{qzd} net. (D-F) the networks shown are the medial axes as
		calculated from the minimal surface, as the curves tracing the center lines
		that are most distance from the minimal surface.
	}
	\label{fig:qtz-qzd-surf-with-medax}
\end{figure*}

We consider the pair of dual interpenetrating 3D networks known as the quartz
(\texttt{qtz}) net and the \texttt{qzd} net \cite{DelgadoF_2_2003}, both
documented in reticular chemistry structure resource (RCSR,
\cite{OKeeffe_2008}), see Figure \ref{fig:qtz-qzd-nets}. 

The \texttt{qtz} net has symmetry group $P6_222$, with identical
four-coordinated vertices, all at symmetry site 3c (in Wyckoff notation), and
with a single type of edge along the two-fold axes with Wyckoff symbol 6j. The
dual net of \texttt{qtz} is the \texttt{qzd} (quartz dual) net, with the same
space group. The \texttt{qzd} net has a single type of vertex, at Wyckoff
position 3a with symmetry $222$, on the six-fold $6_2$ screw axis in vertical
(c) direction. The \texttt{qzd} net has two edge types: a vertical edge along
the $6_2$ screw axis and a horizontal edge along a two-fold rotational symmetry
axis. The nodes of the \texttt{qzd} net are planar and four-coordinated, with
edges at $90^o$. The \texttt{qtz}-\texttt{qzd} structure has one free parameter
(the value of the crystallographic $c/a$), in addition to the lattice parameter
$a$ which only represents an affine scaling.

The \texttt{qtz}-\texttt{qzd} network pair was chosen, following advice by
Michael O'Keeffe, as it represents two dual nets that are not identical, that
have chiral screw symmetries and a free parameter. This was motivated by our
optical work in which we sought chiral sheet-like prototype geometries for
photonic crystals with a tunable chiral pitch.  

Following the philosophy set out by O'Keeffe, Hyde and others, the quartz (or
{\texttt{qtz}}) net and the {\texttt{qzd}} net, shown in Figure 1(a) and (b),
are an adequate choice for the generation of a well-defined dividing surface,
as they are duals (in the notion of Appendix II of \cite{Schoen_1970}) of one
another and of the same symmetry (called proper duals \cite{Blatov_2007}). 

\subsection{Evolution of tubular net to a mesh approximation of the minimal surface}

As the next stage of the construction, a discrete representation
(triangulation) of an interface between the {\texttt{qtz}} and the
{\texttt{qzd}} nets can be generated, which is neither smooth nor
area-minimizing. This can be the tubular surface, or a discretisation thereof,
around one of the two networks; or a culled Voronoi diagram of a set of points
on the two graphs (with all Voronoi faces removed that separate points on the
same graph \cite{de_Campo_2013,Himmelmann2025}).  The resulting triangulated
surface can be evolved towards a minimal surface, by using a conjugate gradient
scheme with a function $f[S]=\int_S (H-H_0)^2 dS$ with $H_0=0$ and $H$ the
point-wise mean curvature. Leaving mathematical details aside, minimisation of
this functional to $f[S]=0$ yields the same minimal surface as does
area-minimisation with a fixed (but here unknown) volume constraint
\cite{KGB_1997}. For a general surface, it is however not guaranteed that a
minimum with constant $H=0$ can be reached by a topology- and/or
symmetry-preserving evolution, however for the \texttt{qtz}-\texttt{qzd}
surface it works (We have used this method for $0.4\le c/a\le 2.9$). 

We have used Brakke's conjugate gradient solver {\tt Surface Evolver}
\cite{Brakke_1992} to determine a surface minimising $f[S]$, starting from a
tubular mesh around the {\texttt{qtz}} net. To a high degree of numerical
accuracy, the resulting surface has $f[S]=0$, with discrete point-wise mean
curvatures verified to vanish. By computer-visual inspection, all symmetries of
the original symmetry group are preserved.

The mesh representation we obtain is of sufficiently high quality to determine
the flat points (the points with Gauss curvature $K=0$) and their nature. In an
earlier pre-print version of this article \cite{preprint} (see in particular
section 4.2), we identified these flat points.

When one constructs the medial surface of the \texttt{qtz} domain and reduces
it to a one-dimensional net, with nodes and edges tracing lines of maximal
distance to the surface, it turns out that that 'skeletal graph' is not the
four-coordinated \texttt{qtz} net, but rather a net, known as \texttt{eta}
\cite{Rosi_2005}, with the same symmetry group with three-coordinated vertices.
See Fig.\ \ref{fig:qtz-qzd-surf-with-medax} (F).

\section{Toric Weierstrass parameterization}\label{sec:Weierstrass-construction-hao}

There is another way to approach the previously described surface. It relates
to Allan Schoen's model of the H'-T surface \cite{Schoen_1970} and gives rise
to our name for this surface as the \textit{gyrating H'-T} surface family.  

Schoen described a surface that he called H'-T \cite{Schoen_1970}. Its
translational unit consists of three ``catenoids'', one bounded by two regular
hexagons and the other two bounded by two equiangular triangles, with
catenoidal axis all in vertical direction (i.e., crystallographic $c$ axis),
see Fig.\ \ref{fig:Schoen-Hprime-T}.   Apart from periodic translations, and
H'-T surface also admits a six-fold rotation symmetry around the vertical axis
through the center of the hexagons, and two-fold rotational symmetries around
the horizontal axes that interchange the two hexagons.  

The skeletal graphs of the H'-T surface are, as per the reticular chemistry
structure resource, the graphs \texttt{bnn} and \texttt{hex} (which,
individually and as a pair, have symmetry group 191, $P6/mmm$), see Fig.\
\ref{fig:Schoen-Hprime-T} (C and D). The space group of the H'-T surface is,
therefore, space group 191 ($P6/mmm$). 

This means that the surface has six-fold rotational symmetry around a vertical
axes (the vertical edges of the \texttt{hex} net) and horizontal two-fold
rotational symmetry axes (the horizontal edges of the \texttt{hex} net). 

The construction of the \textit{gyrating H'-T} surface (separating the
\texttt{qtz}-\texttt{qzd} nets) is based on the following idea: The six-fold
rotational axis of the H'-T surface changes into a six-fold screw axis ($6_2$)
in the gyrating H'-T surface. The two-fold rotations, while preserved, no
longer intersect in the same point on the screw axis, but rather become a
staggered 'ladder' of two-fold axes shown in Fig.\ \ref{fig:qtz-qzd-nets}.

It was observed in~\cite[Lemma~4]{kgb1996} that, as one travels along the
associate family, the rotational symmetries listed above are preserved, except
for the rotational symmetry around the vertical axis, which is reduced to a
screw symmetry.

We define $\cH$ to be the set of embedded TPMSs of genus 4 that admit order-6
screw symmetries around vertical axes and order-2 rotational symmetries around
horizontal axes.  

Note that pure rotations are screw transforms with 0 translation.  So Schoen's
H'-T surfaces belong to $\cH$.  

Our key result is the statement that there exists another 1-parameter family in
$\cH$, different to Schoen's H'-T surface. We derive an explicit Weierstrass
formula for that surface and prove the following theorem:

\begin{theorem}\label{thm:main}
	Apart from Schoen's H'-T surfaces, there is another 1-parameter family of
	TPMS of genus 4 with the following properties:
	\begin{enumerate}
	  \item Each TPMS in the family admits a screw symmetry of order 6 around a
			vertical axis, and rotational symmetries of order 2 around horizontal
			axes.

		\item The family tends to Scherk saddle towers in one limit and to doubly
			periodic Scherk surfaces in the other limit. 
	\end{enumerate}
\end{theorem}

In~\cite{chen2024}, this family was constructed implicitly and proved to be
unique near the saddle tower limit.  This guarantees that the two constructions
we provide (that is, above in terms of the dual \texttt{qtz}-\texttt{qzd} net
and below in terms of the gyrated H'-T surface) give the same surface family.

\subsection{Weierstrass parametrization}\label{sec:weierstrass}

Let $M$ be a TPMS of genus $g$ invariant under the lattice translations
$\Lambda$ (i.e., $M/\Lambda$ is the translational unit cell).
Meeks~\cite{meeks1990} proved that the Gauss map $G$ represents $M/\Lambda$ as
a $(g-1)$-sheeted conformal branched cover of the sphere $\mathbb{S}^2$, and
Riemann-Hurwitz formula implies $4(g-1)$ branch points (counted with
multiplicity) corresponding to the zeros of the Gauss curvature (counted with
multiplicity).


We use the following form of Weierstrass parametrization that maps a point $p$ on the Riemann surface to a point in $\R^3$ and that traces back to
Osserman~\cite{osserman1964},
\begin{multline}\label{eq:weierstrass2}
	p \mapsto \re\int^p (\omega_1, \omega_2, \omega_3)\\
  = \re\int^p \Big(\frac{G^{-1}-G}{2}, \frac{i(G^{-1}+G)}{2},
	1\Big) dh\,.
\end{multline}
On the Riemann surface $\Sigma$ the functions $\omega_1, \omega_2, \omega_3$
must all be holomorphic.  In particular, the holomorphic differential $\omega_3
= dh$ is called the \textit{height differential}.  $G$ denotes (the stereographic
projection of) the Gauss map.  The triple $(\Sigma, G, dh)$ is called
Weierstrass data.

The purpose of this section is to determine the Weierstrass data for surfaces
in $\cH$ from their symmetries.  In particular, $\Sigma$ will be a branched
torus, whose branch points are determined in Lemma~\ref{lem:branchH}.  The
height differential is determined in Lemma~\ref{lem:dh}, and the Gauss map is
explicitly given by \eqref{eq:G}.

\begin{remark}\label{rem:sigma}
  Most textbooks only mention the Weierstrass parameterization where $\Sigma$
  is taken as a simply connected domain in the complex plane.  Now that we
  allow $\Sigma$ to be any Riemann surface, the parameterization becomes more
  flexible and powerful.  Different choices of $\Sigma$ should lead to the same
  result, but particular choices often bring convenience.  For instance, to
  parametrise a TPMS, one could use a branched cover of the sphere as $\Sigma$
  (as was done by Schoen \cite{Schoen_1970}, Meeks\cite{Meeks1975, meeks1990},
  	Fogden \& Hyde \cite{Fogden_i_1992, Fogden_ii_1992, Fogden_iii_1993}, and
		many others, and as we also did in an earlier arXiv version of this manuscript
	\cite{preprint}).  In this paper, as screw symmetries are concerned, we will
	instead use branched cover of tori as $\Sigma$, which allows us to employ
	elliptic functions.
\end{remark}

\subsection{Weierstrass data on tori}

Surfaces in $\cH$ all admit screw symmetries. The following Lemma justifies our
choice of branched tori for $\Sigma$. Here, $M$ is the infinite periodic
surface, $\Lambda$ are the lattice translations and therefore $M/\Lambda$ the
(primitive) translational unit cell. $S$ is the screw symmetry, and therefore
$(M/\Lambda)/S$ the surface patch from which the translational unit is obtained
by applying the screw symmetry $S$.

\begin{lemma}
  \label{lemma:screw-symmetry}
	If a TPMS $M$ of genus 4 admits a screw symmetry $S$, then $(M/\Lambda)/S$
	is of genus one.
\end{lemma}

The simple proof (given in section \ref{sec:proof-of-lemma-screw}) uses the
Riemann--Hurwitz formula
\[
	g = n(g' - 1) + 1 + B/2
\]
that relates the genus $g$ of the translational unit cell $M/\Lambda$ (in our
case $g=4$) to the genus $g'$ of the translational unit cell modulo $S$, i.e.\
$(M/\Lambda)/S$. $n$ is the degree of the quotient map, and $B$ is the total
branching number. The proof simply considers each of the small number of
discrete possibilities. 

The height differential $dh$, being a holomorphic 1-form on the torus, must be
of the form $r e^{-i\theta} dz$.  Varying the modulus $r$ only results in a
scaling.  Varying the argument $\theta$ gives the associate family, so we call
$\theta$ the \textit{associate angle}:

\begin{lemma}\label{lem:dh}
	If a TPMS $M/\Lambda$ with a screw symmetry $S$ is represented on the
	branched cover of the torus $(M/\Lambda)/S$, then up to the scaling, the
	height differential $dh$ must be the lift of $e^{-i\theta}dz$ (note the
	sign!).
\end{lemma}

\subsection{Locating branch points}

If the order of $S$ is prime, the following formula
from~\cite{farkas1992} allows us to calculate the number of fixed points: 
\[
	|\operatorname{fix}(S)| = 2 + \frac{2g - 2g'\operatorname{order}(S)}{\operatorname{order}(S) - 1}.
\]
In particular, a screw symmetry of order 2 fixes exactly six points.  They
correspond to zeroes and poles of $G^2$.  The following lemma follows from the
same argument as in~\cite[Lemmas~3.9, 3.13]{weyhaupt2006}.

\begin{lemma}
	If a TPMS of genus 4 $M/\Lambda$ admits a screw symmetry $S$ of order 2, then
	$G^2$ descends to an elliptic function on the torus $(M/\Lambda)/S$ with
	three simple zeros and three simple poles.
\end{lemma}

We now try to locate the branch points of the covering map for surfaces in
$\cH$.  Since the ramification points on $M/\Lambda$ (i.e.\ points where the
covering map fails to be a local homeomorphism) are all poles and zeros of the
Gauss map, our main tool is naturally Abel's Theorem, which states that the
difference between the sum of poles and the sum of zeros (counting
multiplicity) is a lattice point.

A surface in $\cH$ admits a screw symmetry $S$ of order 6, $S^3$ is then a
screw symmetry of order 2. 

\begin{lemma}[Compare~{\cite[Lemma 3.10]{weyhaupt2006}}]\label{lem:branchH}
	Let $M$ be a TPMS of genus 4 admitting a screw symmetry $S$ of order 6, hence
	parametrized on a branched double cover of the torus $(M/\Lambda)/S^3$.  If
	one branch point is placed at $0$, then the other branch points must be
	placed at 3-division points.
\end{lemma}

Here, a point $p$ in the torus is a \textit{$3$-division point} if $3p = 0$.
Or, equivalently, if we see the torus as the quotient of the complex plane over
a lattice, then a point is a $3$-division point if $3p$ is a lattice point. See
Appendix \ref{proof-of-lem:branchH} for the proof of this lemma.

Let the quotient torus be spanned by $1$ and $\tau \in \C$.  Without loss of
generality, we may then assume that the zeros of $G^2$ are at $0$, $1/3$, and
$2/3$, while the poles of $G^2$ are at $\tau/3$, $\tau/3+1/3$, and
$\tau/3+2/3$.  See Figure~\ref{fig:branchedTorus}. 

\begin{figure*}
  \includegraphics[width=\textwidth]{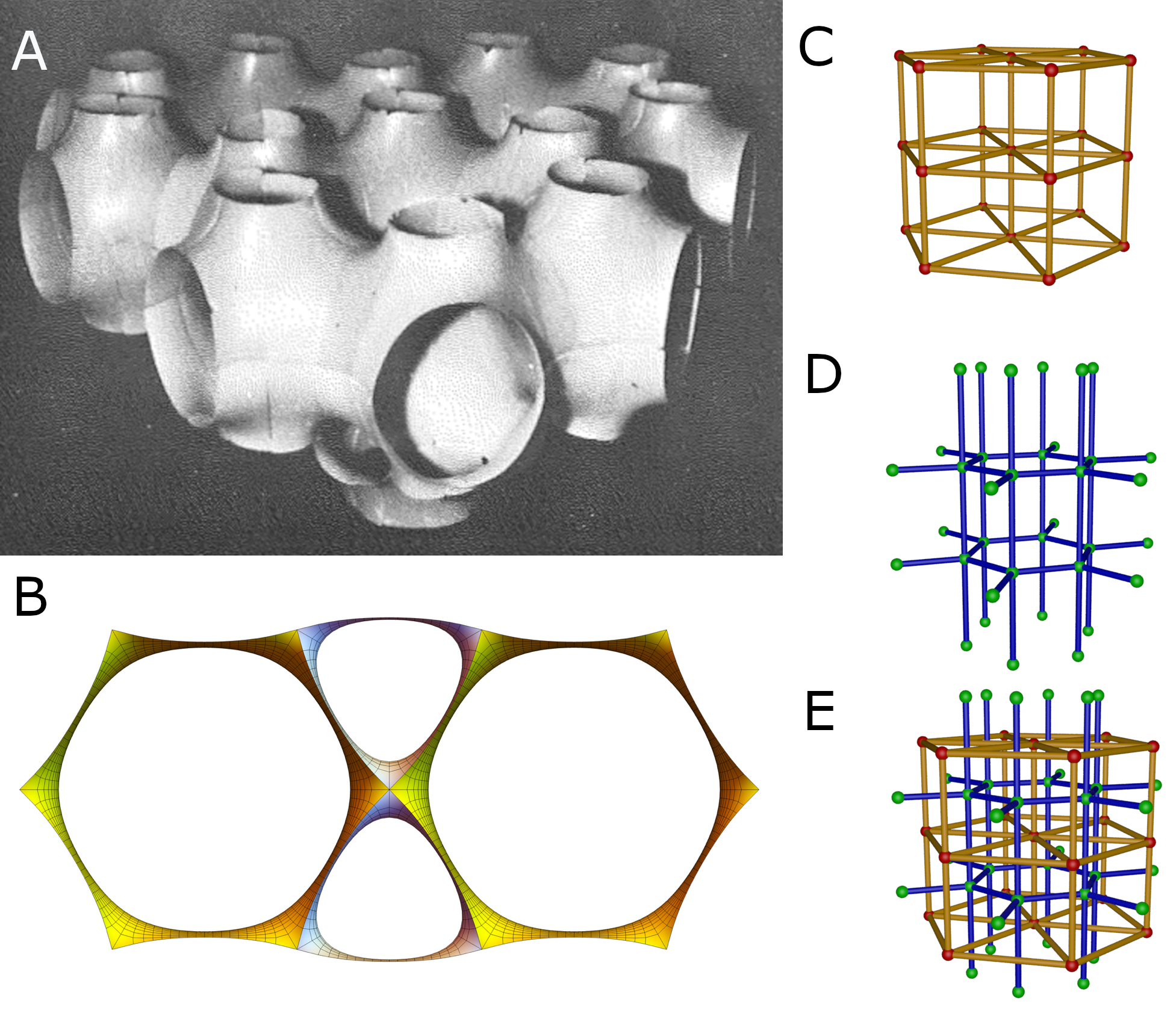}
  \caption{ \label{fig:Schoen-Hprime-T} Schoen's H'-T surface: (A) Photograph
    of a plasticine model created by Alan Schoen. (B) The H'-T surface can be
    constructed by copy-translations from a hexagonal catenoidal patch and
    two adjacent triangular catenoidal necks, all aligned in the vertical
    axis (that is, the crystallographic $c$ axis). (C) The hexagonal lattice
    (\texttt{hex}) representing the skeletal graphs on one side of the
    surface (D) The \texttt{bnn} net representing the other skeletal graph.
    (E) The dual pair of \texttt{hex} and \texttt{bnn} nets.  (Image (A)
    reproduced from Figure 13 of \cite{Schoen_1970}).
  }
\end{figure*}

\begin{figure*}
	\includegraphics[width=0.6 \textwidth]{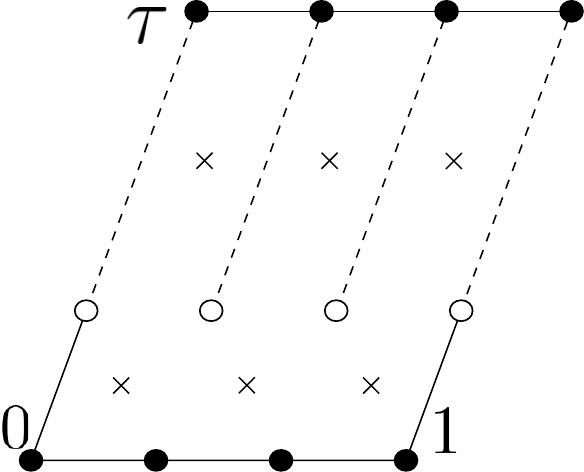}
	\caption{
    \label{fig:branchedTorus} Branched torus on which our surface is
    parametrized.  Dashed segments are branched cuts.  Solid circles are the
    zeros of $G^2$, empty circles are the poles, and $\times$ represent
    symmetry centers, corresponding to fixed points on the surface under
    order-$2$ rotations with horizontal axis.
	} 
\end{figure*}

Then $G^2$ has the explicit form
\[
	G^2(z) = \rho \frac{\theta(z;\tau) \theta(z-\frac{1}{3};\tau) \theta(z+\frac{1}{3};\tau)}{\theta(z+\frac{2\tau}{3};\tau) \theta(z - \frac{\tau}{3} - \frac{1}{3};\tau) \theta(z - \frac{\tau}{3} + \frac{1}{3};\tau)},
\]
where $\theta(z; \tau)$ is the Jacobi Theta function.

The factor $\rho$ is called the L\'opez--Ros factor \cite{lopez1991}.  Varying
its argument only results in a rotation of the surface in the space, hence only
the norm $|\rho|$ concerns us.  We choose $\rho$ so that $G^2 = 1$ at the
rotation center $\tau/6$. 

Equivalently, by \cite[(9.6.6)]{lawden1989}, we can write
\begin{equation}\label{eq:G}
	G^2(z) = \rho' e^{2 i \pi z} \frac{\theta(3z;3\tau)}{\theta(3z-\tau;3\tau)},
\end{equation}
where $\rho' = -e^{i\pi\tau/3}$ is again chosen so that $G^2(\tau/6) = 1$.

\subsection{Twisted catenoids}\label{sec:period}

We choose three branch cuts along the straight segments from $\tau/3 + k/3$ to
$\tau + k/3$, $k = 0, 1, 2$.  Assume $dh = e^{-i\pi/2} = -idz$ for the moment.
We now study the image of the Weierstrass parametrization

From the quasi-periodicity of $\theta$ function
\[
	\theta(z + (m+n\tau);\tau) = (-1)^{m+n} e^{-in^2\pi\tau} e^{-2in\pi z} \theta(z;\tau),
\]
we note that
\[
	G^2(z + 1/3) = -e^{2 i \pi / 3} G^2(z).
\]

By the same argument as in~\cite{chen2021}, we know the following
\begin{enumerate}
	\item For each branch of the branched torus, the part $0 < \im z < \im
		\tau/3$ is mapped by the Weierstrass representation to a twisted catenoids
		bounded by curved triangles.  These triangles are congruent, right-angled,
		and lie in horizontal planes at height $0$ and $\im\tau/3$, respectively.
		Moreover, they share a rotational symmetry of order $3$ around vertical
		axis.

	\item The part $\im\tau/3 < \im z < \im\tau$ of the branched torus is mapped
		to a twisted catenoid bounded by curved hexagons.  These hexagons are
		congruent, right-angled, and lie in horizontal planes at height $\im\tau/3$
		and $\im\tau$, respectively.  Moreover, they share a rotational symmetry of
		order $6$ around a vertical axis.

 	\item Because of our choice of the L\'opez--Ros factor $\rho$, all these
		twisted catenoids admit rotational symmetries of order $2$ around
		horizontal axes that exchange its boundaries.
\end{enumerate}

See Figure~\ref{fig:twistedCatenoid} for a typical example.

\begin{figure*}
	\includegraphics[width=0.6 \textwidth]{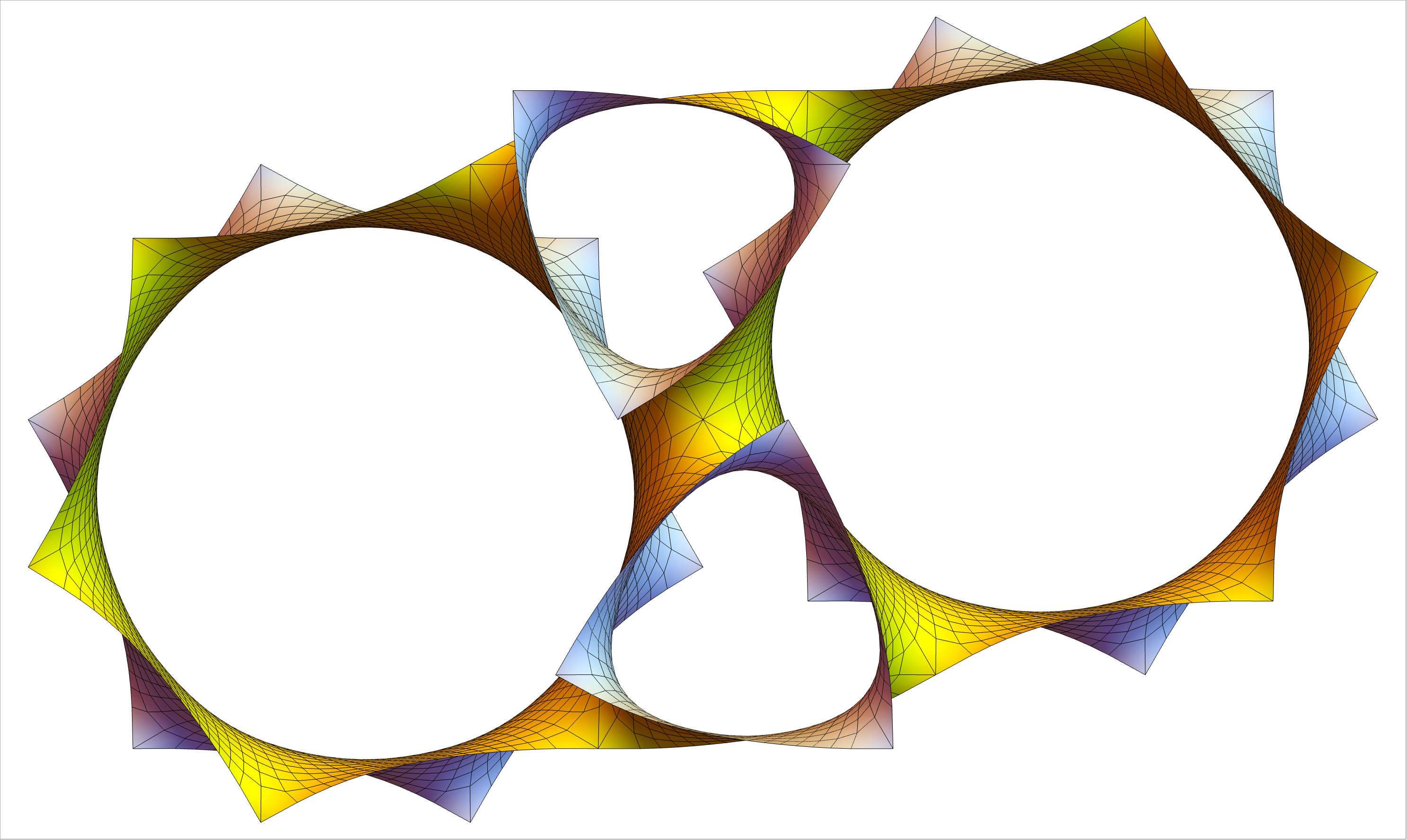}
	\caption{ \label{fig:twistedCatenoid} Image of Weierstrass representation
		assuming bonnet angle $\theta = \pi/2$, seen from below, showing two
		twisted triangular catenoids (corresponding to the lower third of the
		branched torus), and two twisted hexagonal catenoids (one of these
			corresponds to the upper two thirds of the branched torus, the other one
		is a lattice translation included for clarity).
	}
\end{figure*}

\subsection{Period problem}

As one travels from the twisted catenoids along the associate family, the
twisted catenoid opens up into a ribbon bounded by triangular or hexagonal
helices.  For adjacent ribbons to ``fit exactly into each
other''~\cite{kgb1996}, we need that the poles and zeros of the Gauss map are
(i) aligned along vertical lines over the vertices of a hexagon-triangle
tiling, and (ii) alternatingly arranged and the vertical width of the hexagonal
ribbon doubles that of the triangular ribbon.

The ratio of the vertical width is actually a consequence of the positions of
the branch points.  The requirements listed above form a sufficient condition
for the immersion.

To be more precise, we define the pitch of a helix to be the increase of height
after the helix makes a full turn.  We need that the pitch of each triangular
helix is three times the minimum vertical distance between the poles and the
zeros.  As a consequence, the pitch of each hexagonal helix, which doubles the
pitch of triangular helices, must be six times the minimum vertical distance
between the poles and the zeros.

This means that the integral of the height differential $dh$
from $0$ to $1$ triples the integral from $0$ to $\tau/3$.  Or
equivalently, the integral from $0$ to $1/3$ equals the integral from $0$ to
$\tau/3$.  This can be easily achieved by adjusting the associate angle
$\theta$ to (compare~\cite[Definition~4.2]{weyhaupt2008})
\[
	\theta = \theta_v(\tau) = \arg(\tau-1) - \pi/2.
\]

\medskip

We now calculate the associate angle in another way, using the fact that the
images of $0$ and $\tau$ are vertically aligned, i.e.\ have the same
horizontal coordinates.

First note that, because $\tau/6$ is a center of symmetry, we have
\[
	\int_0^{\tau/3} dz/G = \int_0^{\tau/3} dz \cdot G =: \psi(\tau)
\]
We may place the image of $0$ at the origin.  First look at the surface with
$\theta=0$, hence $dh = dz$.  Then the horizontal coordinates of the image of
$\tau/3$ are
\[
	\re \int_0^{\tau/3} \Big(\frac{1}{2}\big(\frac{1}{G}-G\big), \frac{i}{2}\big(\frac{1}{G}+G\big)\Big)dz = (0, -\im \psi(\tau)).
\]
Then we look at the surface with $\theta=\pi/2$, hence $dh = e^{-i\pi/2} dz =
-i dz$ (the conjugate surface).  Then the coordinates are
\[
	\re \int_0^{\tau/3} \Big(-\frac{i}{2}\big(\frac{1}{G}-G\big), \frac{1}{2}\big(\frac{1}{G}+G\big)\Big)dz = (0, \re \psi(\tau)).
\]
So for the surface with associate angle $\theta$, the first coordinate is
always $0$, while the second coordinate
\[
	- \cos\theta \im\psi(\tau) + \sin\theta \re\psi(\tau)
\]
vanishes when
\[
	\theta = \theta_h(\tau) := \arg\psi(\tau).
\]

\medskip

We have shown that

\begin{lemma}
	The Weierstrass data given by Lemmas~\ref{lem:dh} and~\eqref{eq:G} define an
	immersion if and only if
	\[
		\theta_h(\tau)=\theta_v(\tau),
	\]
	or more explicitly,
	\begin{equation}\label{eq:period}
		\arg\psi(\tau) = \arg(\tau-1)-\pi/2
	\end{equation}
\end{lemma}

\section{Existence proof}\label{sec:proof}

\begin{figure*}
  \includegraphics[width=0.6 \textwidth]{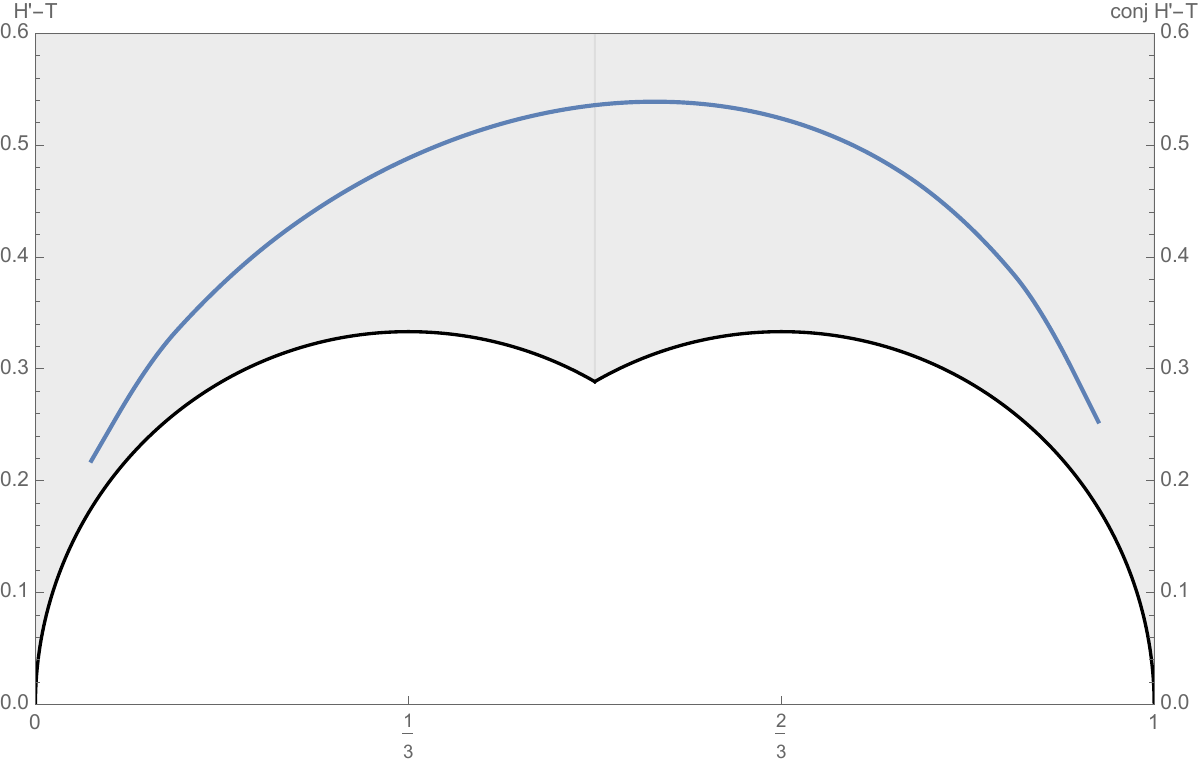}
	\caption{\label{fig:tau}
		Solution to the period problem~\eqref{eq:period} with $0 < \re\tau < 1$.
	}
\end{figure*}

In Figure~\ref{fig:tau}, we show the numerical solutions to~\eqref{eq:period}
with $0 < \re \tau < 1$, accompanied by two circular arcs.  Our task is to
prove the existence of the continuous 1-parameter solution curve that we see in
the picture.  Let the shaded domain in the figure be denoted by $\Omega_t := \{
\tau \mid \im\tau>0, 0<\re\tau<1, |\tau - 1/3| > 1/3, |\tau - 2/3| > 1/3\}$.

\begin{proposition}\label{existence-proof-proposition}
	There exists a continuous 1-dimensional curve of $\tau$ in $\Omega_t$ that
	solves~\eqref{eq:period}.  This curve tends to $0$ at one end and to $1$ at
	the other end.  Moreover, the TPMSs of genus 4 represented by points on the
	curve are all embedded.
\end{proposition}

The proof for this proposition is found in Appendix
\ref{suppmatsec:proof-of-existence}. In its essence, it constructs a solution
to the period problem in terms of the parameter $\tau$. Surfaces with the same
symmetry near the saddle tower limit have been constructed in~\cite{chen2024}
by a gluing constructions. Those surfaces have been proven to be embedded and
unique, and hence must belong to the curve above near the limit 0.  Their
embeddedness then ensures the embeddedness of all TPMSs on the curve.  This
follows from~\cite[Proposition~5.6]{weyhaupt2006}, which was essentially proved
in~\cite{meeks1990}.

In Figure~\ref{fig:gallery}, we show four gyrating H'-T surfaces.

\begin{figure*}
	\begin{center}
		\includegraphics[height=0.35\textwidth]{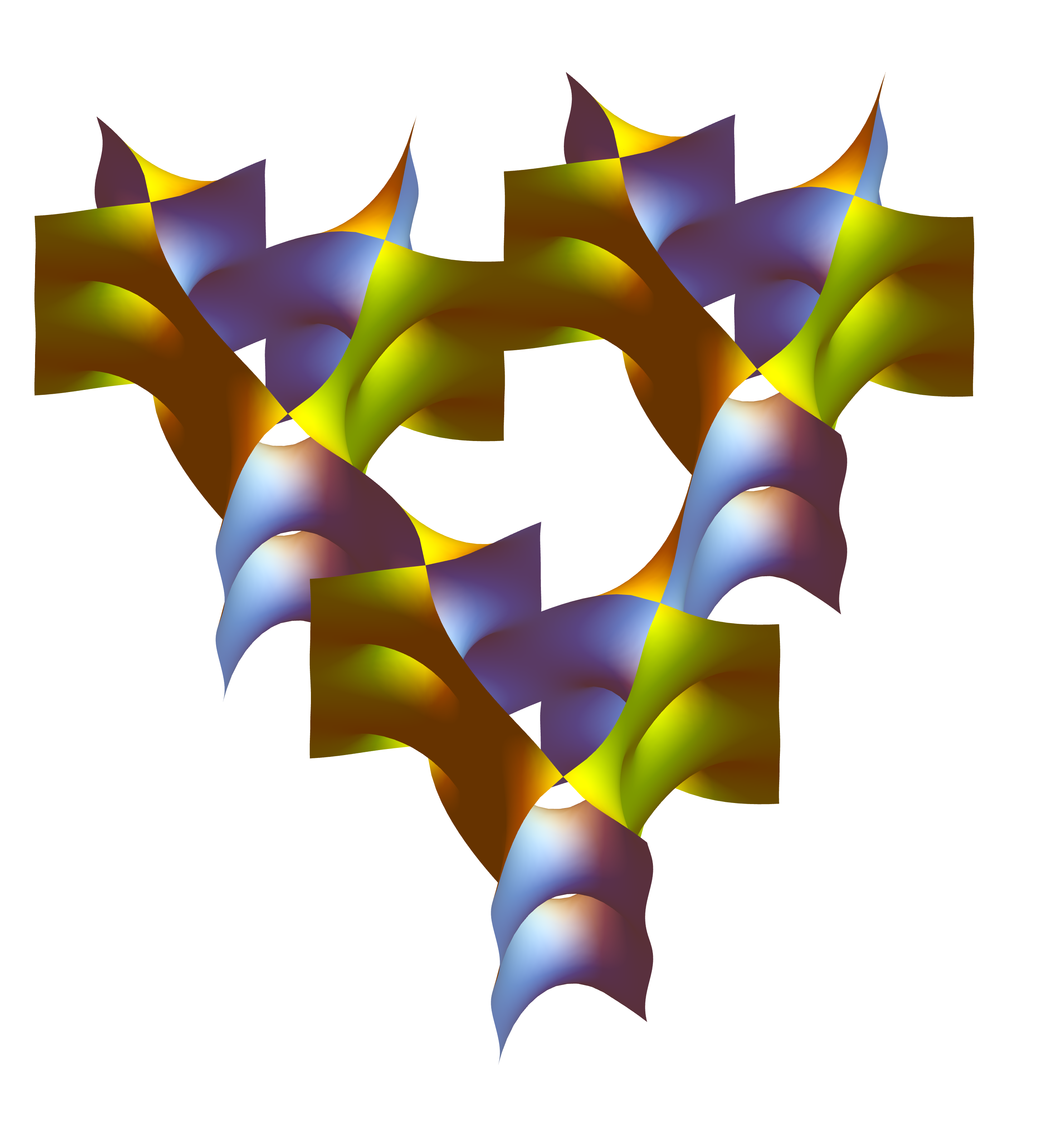}
		\includegraphics[height=0.35\textwidth]{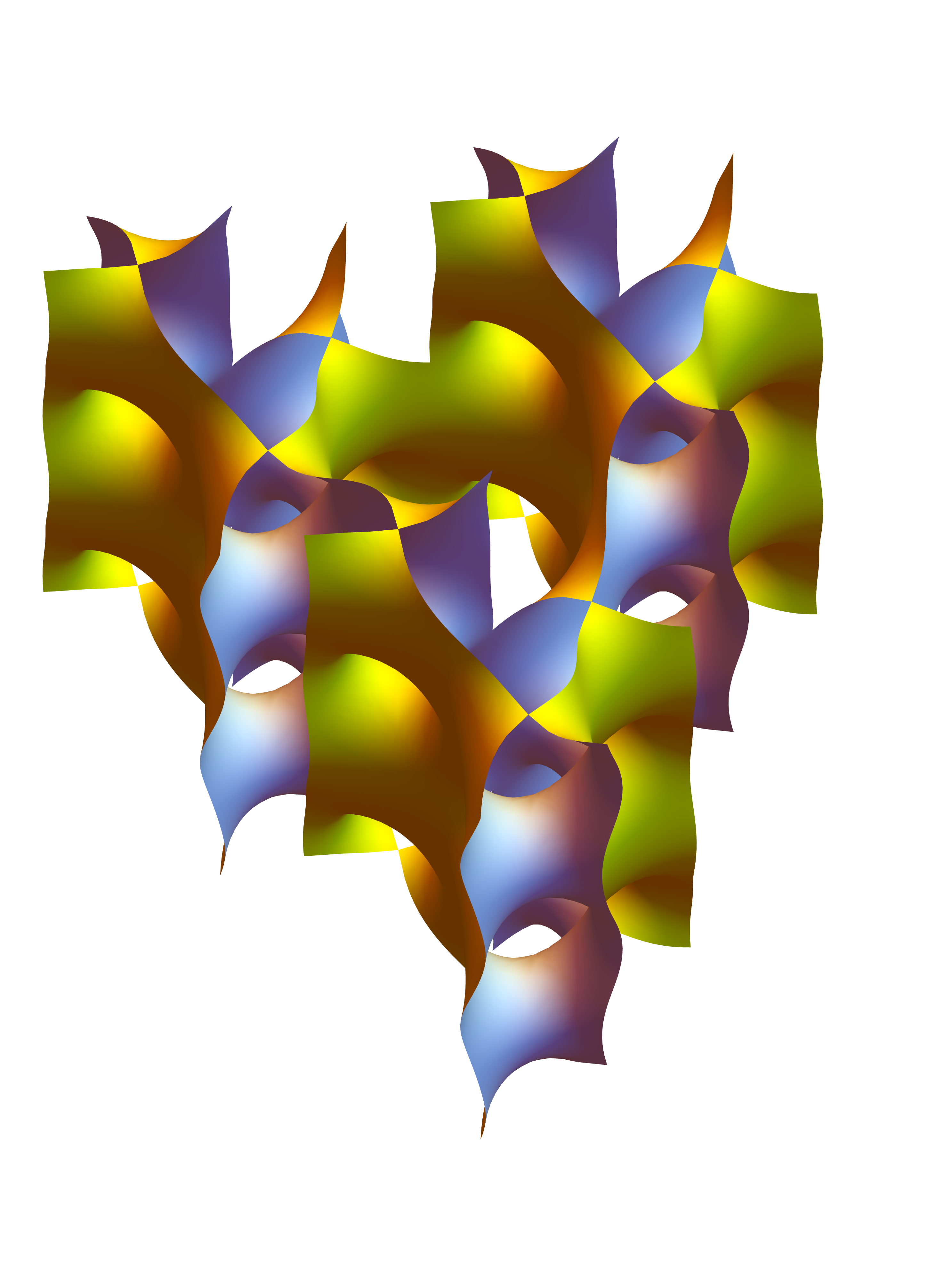}
		\includegraphics[height=0.35\textwidth]{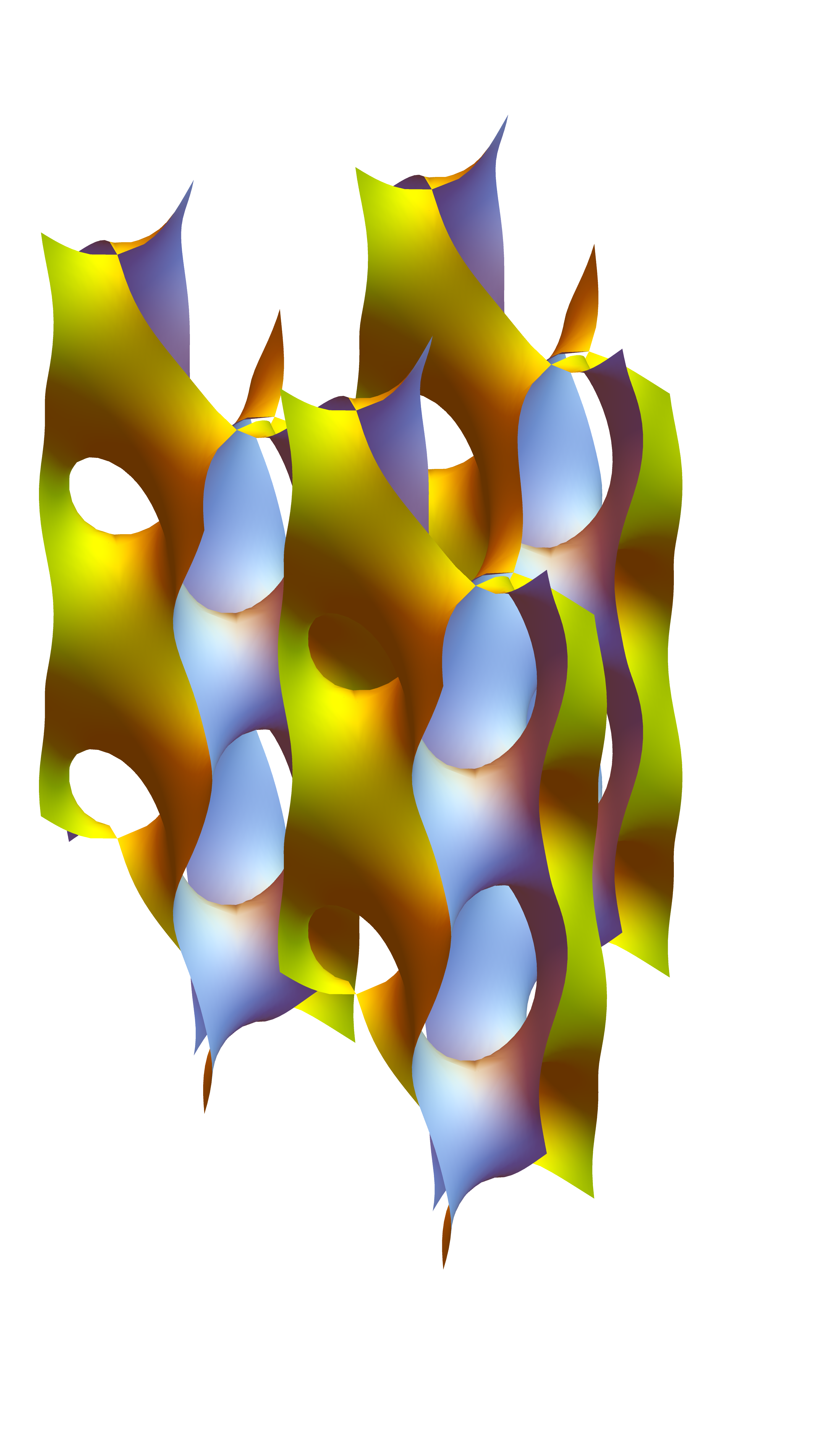}
		\includegraphics[height=0.35\textwidth]{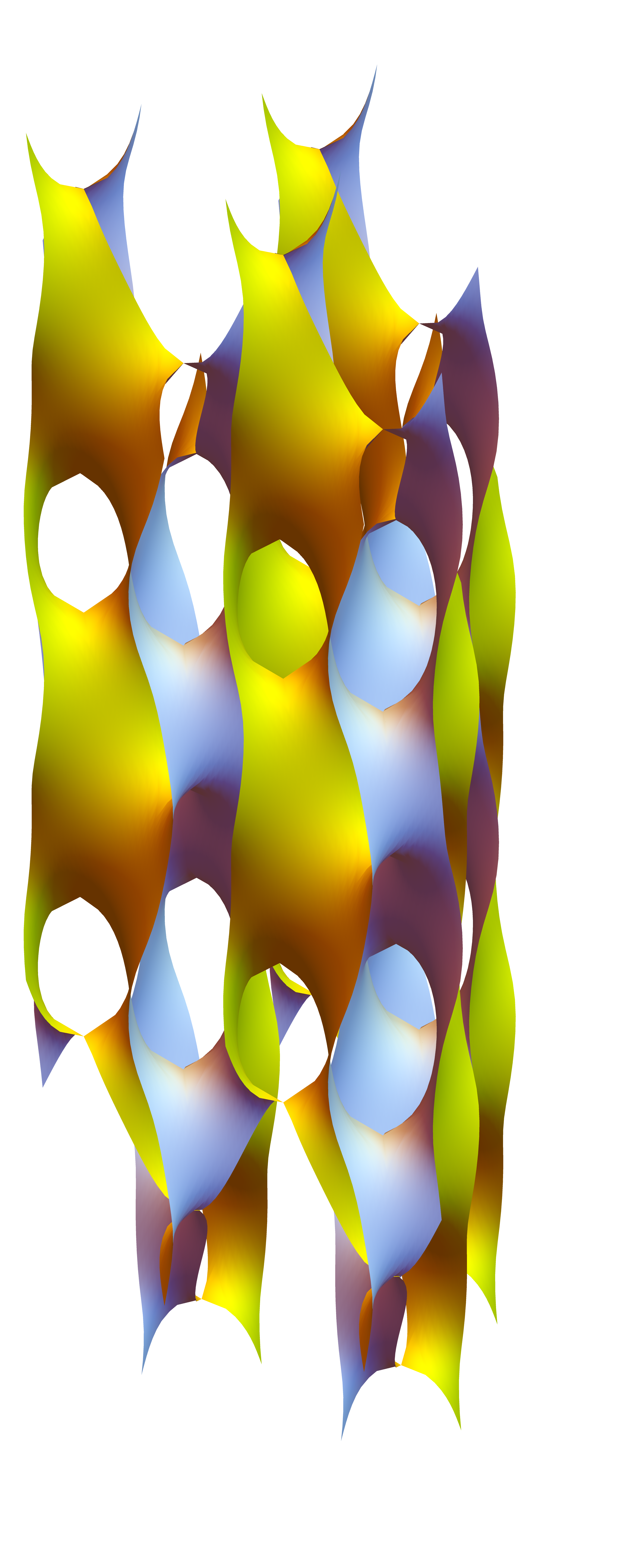}
		\caption{
			Gyrating H'-T surfaces with $\re\tau = 0.2, 0.4, 0.6, 0.8$, from left to right.
		}
		\label{fig:gallery}
  \end{center} 
\end{figure*}

\subsection{Acknowledgement}

We are grateful to Michael O'Keeffe for his suggestion of the
\texttt{qtz}-\texttt{qzd} as a chiral pair of networks with a tunable pitch. We
are grateful to Robert Kusner and Karsten Gro{\ss}e-Brauckmann for early
discussions of the period problem.

\bibliography{References}

\newcommand{\etalchar}[1]{$^{#1}$}
\begin{thebibliography}{dCDFHO13}

\bibitem[BDFOP07]{Blatov_2007}
Vladislav~A. Blatov, Olaf Delgado-Friedrichs, Michael O'Keeffe, and Davide~M.
  Proserpio.
\newblock Three-periodic nets and tilings: natural tilings for nets.
\newblock {\em Acta Crystallographica Section A Foundations of
  Crystallography}, 63(5):418--425, August 2007.

\bibitem[Bra92]{Brakke_1992}
Kenneth~A. Brakke.
\newblock The surface evolver.
\newblock {\em Experimental Mathematics}, 1(2):141--165, January 1992.

\bibitem[Bra25]{BrakkeTPMSWebArchive}
Ken Brakke.
\newblock Triply-periodic minimal surfaces.
\newblock Technical report,
  https://kenbrakke.com/evolver/examples/periodic/periodic.html, 2025.

\bibitem[CF22]{chen2022}
Hao Chen and Daniel Freese.
\newblock Helicoids and vortices.
\newblock {\em Proc. A.}, 478(2267):Paper No. 20220431, 18, 2022.

\bibitem[Che21a]{Chen_2021}
Hao Chen.
\newblock Existence of the tetragonal and rhombohedral deformation families of
  the gyroid.
\newblock {\em Indiana University Mathematics Journal}, 70(4):1543–1576,
  2021.

\bibitem[Che21b]{chen2021}
Hao Chen.
\newblock Existence of the tetragonal and rhombohedral deformation families of
  the gyroid.
\newblock {\em Indiana Univ. Math. J.}, 70(4):1543--1576, 2021.

\bibitem[CT24]{chen2024}
Hao Chen and Martin Traizet.
\newblock Gluing {K}archer-{S}cherk saddle towers {I}: triply periodic minimal
  surfaces.
\newblock {\em J. Reine Angew. Math.}, 808:1--47, 2024.

\bibitem[CW21a]{Chen2021a}
Hao Chen and Matthias Weber.
\newblock A new deformation family of schwarz’ d surface.
\newblock {\em Transactions of the American Mathematical Society},
  374(4):2785–2803, January 2021.

\bibitem[CW21b]{Chen2021b}
Hao Chen and Matthias Weber.
\newblock An orthorhombic deformation family of {S}chwarz' {H} surfaces.
\newblock {\em Trans. Amer. Math. Soc.}, 374(3):2057--2078, 2021.

\bibitem[dCDFHO13]{de_Campo_2013}
Liliana de~Campo, Olaf Delgado-Friedrichs, Stephen~T. Hyde, and Michael
  O'Keeffe.
\newblock Minimal nets and minimal minimal surfaces.
\newblock {\em Acta Crystallographica Section A Foundations of
  Crystallography}, 69(5):483--489, July 2013.

\bibitem[DFOY03]{DelgadoF_2_2003}
Olaf Delgado~Friedrichs, Michael O'Keeffe, and Omar~M. Yaghi.
\newblock The cdso4, rutile, cooperite and quartz dual nets: interpenetration
  and catenation.
\newblock {\em Solid State Sciences}, 5(1):73 -- 78, 2003.
\newblock Dedicated to Sten Andersson for his scientific contribution to Solid
  State and Structural Chemistry.

\bibitem[DGMG25]{DimitriyevGrason2025}
Michael~S. Dimitriyev, Benjamin~R. Greenvall, Rejoy Mathew, and Gregory~M.
  Grason.
\newblock Not even metastable: Cubic double-diamond in diblock copolymer melts.
\newblock {\em ACS Macro Letters}, 14(9):1291--1298, 2025.
\newblock PMID: 40864671.

\bibitem[{\relax DLMF}]{DLMF}
{\it NIST Digital Library of Mathematical Functions}.
\newblock \texttt{http://dlmf.nist.gov/}, Release 1.0.19 of 2018-06-22.
\newblock F.~W.~J. Olver, A.~B. {Olde Daalhuis}, D.~W. Lozier, B.~I. Schneider,
  R.~F. Boisvert, C.~W. Clark, B.~R. Miller and B.~V. Saunders, eds.

\bibitem[FH92a]{Fogden_i_1992}
A.~Fogden and S.~T. Hyde.
\newblock Parametrization of triply periodic minimal surfaces. i. mathematical
  basis of the construction algorithm for the regular class.
\newblock {\em Acta Crystallographica Section A Foundations of
  Crystallography}, 48(4):442--451, July 1992.

\bibitem[FH92b]{Fogden_ii_1992}
A.~Fogden and S.~T. Hyde.
\newblock Parametrization of triply periodic minimal surfaces. ii. regular
  class solutions.
\newblock {\em Acta Crystallographica Section A Foundations of
  Crystallography}, 48(4):575--591, July 1992.

\bibitem[FH99]{fogden1999}
Andrew Fogden and Stephan~T. Hyde.
\newblock Continuous transformations of cubic minimal surfaces.
\newblock {\em The European Physical Journal B-Condensed Matter and Complex
  Systems}, 7(1):91--104, 1999.

\bibitem[FK92]{farkas1992}
H.~M. Farkas and I.~Kra.
\newblock {\em Riemann surfaces}, volume~71 of {\em Graduate Texts in
  Mathematics}.
\newblock Springer-Verlag, New York, second edition, 1992.

\bibitem[Fog93]{Fogden_iii_1993}
A.~Fogden.
\newblock Parametrization of triply periodic minimal surfaces. iii. general
  algorithm and specific examples for the irregular class.
\newblock {\em Acta Crystallographica Section A Foundations of
  Crystallography}, 49(3):409--421, May 1993.

\bibitem[FW09]{Fujimori_2009}
Shoichi Fujimori and Matthias Weber.
\newblock Triply periodic minimal surfaces bounded by vertical symmetry planes.
\newblock {\em manuscripta mathematica}, 129(1):29--53, May 2009.

\bibitem[GB97]{KGB_1997}
Karsten Gro{\ss}e-Brauckmann.
\newblock Gyroids of constant mean curvature.
\newblock {\em Experimental Mathematics}, 6(1):33--50, January 1997.

\bibitem[GBW96]{kgb1996}
Karsten Gro{\ss}e-Brauckmann and Meinhard Wohlgemuth.
\newblock The gyroid is embedded and has constant mean curvature companions.
\newblock {\em Calc. Var. Partial Differential Equations}, 4(6):499--523, 1996.

\bibitem[GK00]{Gandy_i_2000}
Paul~J.F Gandy and Jacek Klinowski.
\newblock Exact computation of the triply periodic g (`gyroid') minimal
  surface.
\newblock {\em Chemical Physics Letters}, 321(5):363 -- 371, 2000.

\bibitem[HAL{\etalchar{+}}97]{Hyde_1997}
S.~T Hyde, S.~Andersson, K.~Larsson, Z.~Blum, T.~Landh, S.~Lidin, and B.W.
  Ninham.
\newblock {\em The language of shape: the role of curvature in condensed
  matter: physics, chemistry and biology}.
\newblock Elsevier, Amsterdam, 1997.

\bibitem[HdCO09]{Hyde_2009}
Stephen~T. Hyde, Liliana de~Campo, and Christophe Oguey.
\newblock Tricontinuous mesophases of balanced three-arm ``star polyphiles''.
\newblock {\em Soft Matter}, 5(14):2782, 2009.

\bibitem[HOP08]{hydeUbiquitousSRS}
Stephen~T. Hyde, Michael O'Keeffe, and Davide~M. Proserpio.
\newblock A short history of an elusive yet ubiquitous structure in chemistry,
  materials, and mathematics.
\newblock {\em Angewandte Chemie International Edition}, 47(42):7996--8000,
  2008.

\bibitem[HPK{\etalchar{+}}25]{Himmelmann2025}
M~Himmelmann, MC~Pedersen, MA~Klatt, PWA {Sch\"onh\"ofer}, ME~Evans, and
  GE~{Schr\"oder-Turk}.
\newblock Amorphous bicontinuous minimal surface models and the superior
  gaussian curvature uniformity of diamond, primitive and gyroid surfaces.
\newblock {\em Proc.\ R.\ Soc.\ A}, 2025.

\bibitem[HST24]{SchoenObituary}
Stephen~T. Hyde and Gerd~E Schroeder-Turk.
\newblock Alan hugh schoen.
\newblock {\em Physics Today}, 2024(01):.ykfu, January 2024.

\bibitem[Hyd90]{Hyde_1990}
S.~T. Hyde.
\newblock Curvature and the global structure of interfaces in surfactant-water
  systems.
\newblock {\em Le Journal de Physique Colloques}, 51(C7):C7--209--C7--228,
  December 1990.

\bibitem[KF88]{Koch_1988}
Elke Koch and Werner Fischer.
\newblock On 3-periodic minimal surfaces with non-cubic symmetry.
\newblock {\em Zeitschrift f{\"u}r Kristallographie - Crystalline Materials},
  183(1-4), January 1988.

\bibitem[KP96]{Karcher_1996}
H.~Karcher and K.~Polthier.
\newblock Construction of triply periodic minimal surfaces.
\newblock {\em Philosophical Transactions of the Royal Society A: Mathematical,
  Physical and Engineering Sciences}, 354(1715):2077--2104, September 1996.

\bibitem[Law89]{lawden1989}
Derek~F. Lawden.
\newblock {\em Elliptic functions and applications}, volume~80 of {\em Applied
  Mathematical Sciences}.
\newblock Springer-Verlag, New York, 1989.

\bibitem[LM03]{Lord_2003}
Eric~A. Lord and Alan~L. Mackay.
\newblock Periodic minimal surfaces of cubic symmetry.
\newblock {\em Current Science}, 85(3):346 -- 362, 2003.

\bibitem[LR91]{lopez1991}
Francisco~J. L\'{o}pez and Antonio Ros.
\newblock On embedded complete minimal surfaces of genus zero.
\newblock {\em J. Differential Geom.}, 33(1):293--300, 1991.

\bibitem[Mee90]{meeks1990}
William~H. Meeks, III.
\newblock The theory of triply periodic minimal surfaces.
\newblock {\em Indiana Univ. Math. J.}, 39(3):877--936, 1990.

\bibitem[MSSTM18]{preprint}
Shashank~Ganesh Markande, Matthias Saba, Gerd Schroeder-Turk, and Elisabetta~A.
  Matsumoto.
\newblock A chiral family of triply-periodic minimal surfaces derived from the
  quartz network, 2018.

\bibitem[MSTM12]{Mickel2012}
Walter Mickel, Gerd~E. Schr\"{o}der-Turk, and Klaus Mecke.
\newblock Tensorial minkowski functionals of triply periodic minimal surfaces.
\newblock {\em Interface Focus}, 2(5):623–633, June 2012.

\bibitem[OPRY08]{OKeeffe_2008}
Michael O'Keeffe, Maxim~A. Peskov, Stuart~J. Ramsden, and Omar~M. Yaghi.
\newblock The reticular chemistry structure resource (rcsr) database of, and
  symbols for, crystal nets.
\newblock {\em Accounts of Chemical Research}, 41(12):1782--1789, 2008.
\newblock PMID: 18834152.

\bibitem[Oss64]{osserman1964}
Robert Osserman.
\newblock Global properties of minimal surfaces in {$E^{3}$} and {$E^{n}$}.
\newblock {\em Ann. of Math. (2)}, 80:340--364, 1964.

\bibitem[PP93]{Pinkall_1993}
Ulrich Pinkall and Konrad Polthier.
\newblock Computing discrete minimal surfaces and their conjugates.
\newblock {\em Experimental Mathematics}, 2(1):15--36, January 1993.

\bibitem[RKE{\etalchar{+}}05]{Rosi_2005}
Nathaniel~L. Rosi, Jaheon Kim, Mohamed Eddaoudi, Banglin Chen, Michael
  O'Keeffe, and Omar~M. Yaghi.
\newblock Rod packings and metal-organic frameworks constructed from rod-shaped
  secondary building units.
\newblock {\em Journal of the American Chemical Society}, 127(5):1504--1518,
  2005.
\newblock PMID: 15686384.

\bibitem[SC89]{Sadoc_1989}
J~F. Sadoc and J.~Charvolin.
\newblock Infinite periodic minimal surfaces and their crystallography in the
  hyperbolic plane.
\newblock {\em Acta Crystallographica Section A Foundations of
  Crystallography}, 45(1):10--20, January 1989.

\bibitem[Sch70]{Schoen_1970}
A.~H. Schoen.
\newblock Infinite periodic minimal surfaces without self-intersections.
\newblock {\em NASA Technical Note}, 1970.

\bibitem[Sch12]{Schoen_2012}
Alan~H. Schoen.
\newblock Reflections concerning triply-periodic minimal surfaces.
\newblock {\em Interface Focus}, 2(5):658 -- 668, 2012.

\bibitem[SKM{\etalchar{+}}07]{ShearmanSeddonCMC2007}
Gemma~C. Shearman, Bee~J. Khoo, Mary-Lynn Motherwell, Kenneth~A. Brakke, Oscar
  Ces, Charlotte~E. Conn, John~M. Seddon, and Richard~H. Templer.
\newblock Calculations of and evidence for chain packing stress in inverse
  lyotropic bicontinuous cubic phases.
\newblock {\em Langmuir}, 23(13):7276--7285, 2007.
\newblock PMID: 17503862.

\bibitem[STdCE{\etalchar{+}}13]{SchrderTurk2013}
Gerd~E. Schr\"{o}der-Turk, Liliana de~Campo, Myfanwy~E. Evans, Matthias Saba,
  Sebastian~C. Kapfer, Trond Varslot, Karsten Grosse-Brauckmann, Stuart
  Ramsden, and Stephen~T. Hyde.
\newblock Polycontinuous geometries for inverse lipid phases with more than two
  aqueous network domains.
\newblock {\em Faraday Discuss.}, 161:215–247, 2013.

\bibitem[STFH06]{schroderturk2006}
Gerd~E Schr{\"o}der-Turk, Andrew Fogden, and Stephen~T Hyde.
\newblock Bicontinuous geometries and molecular self-assembly: comparison of
  local curvature and global packing variations in genus-three cubic,
  tetragonal and rhombohedral surfaces.
\newblock {\em The European Physical Journal B-Condensed Matter and Complex
  Systems}, 54(4):509--524, 2006.

\bibitem[STVdC{\etalchar{+}}11]{SchrderTurk2011}
Gerd~E. Schr\"{o}der-Turk, Trond Varslot, Liliana de~Campo, Sebastian~C.
  Kapfer, and Walter Mickel.
\newblock A bicontinuous mesophase geometry with hexagonal symmetry.
\newblock {\em Langmuir}, 27(17):10475–10483, August 2011.

\bibitem[Tra08]{traizet2008}
Martin Traizet.
\newblock On the genus of triply periodic minimal surfaces.
\newblock {\em J. Differential Geom.}, 79(2):243--275, 2008.

\bibitem[vSN91]{vonSchnering1991}
H.~G. von Schnering and R.~Nesper.
\newblock Nodal surfaces of fourier series: Fundamental invariants of
  structured matter.
\newblock {\em Zeitschrift f\"{u}r Physik B Condensed Matter}, 83(3):407–412,
  October 1991.

\bibitem[WCZ{\etalchar{+}}24]{Wang2024}
Shuqi Wang, Hao Chen, Tianyu Zhong, Quanzheng Deng, Shaobo Yang, Yuanyuan Cao,
  Yongsheng Li, and Lu~Han.
\newblock Tetragonal gyroid structure from symmetry manipulation: A brand-new
  member of the gyroid surface family.
\newblock {\em Chem}, 10(5):1406–1424, May 2024.

\bibitem[Wey06]{weyhaupt2006}
Adam~G. Weyhaupt.
\newblock {\em New families of embedded triply periodic minimal surfaces of
  genus three in euclidean space}.
\newblock ProQuest LLC, Ann Arbor, MI, 2006.
\newblock Thesis (Ph.D.)--Indiana University.

\bibitem[Wey08]{weyhaupt2008}
Adam~G. Weyhaupt.
\newblock Deformations of the gyroid and {L}idinoid minimal surfaces.
\newblock {\em Pacific J. Math.}, 235(1):137--171, 2008.

\bibitem[WHM75]{Meeks1975}
III William H.~Meeks.
\newblock {\em The Geometry and the Conformal Structure of Triply Periodic
  Minimal Surfaces in R3}.
\newblock PhD thesis, University of California Berkeley, 1975.

\end{thebibliography}
\bibliographystyle{alpha}

\appendix

\section{Use of \texttt{Surface Evolver} to obtain approximate minimal surface}

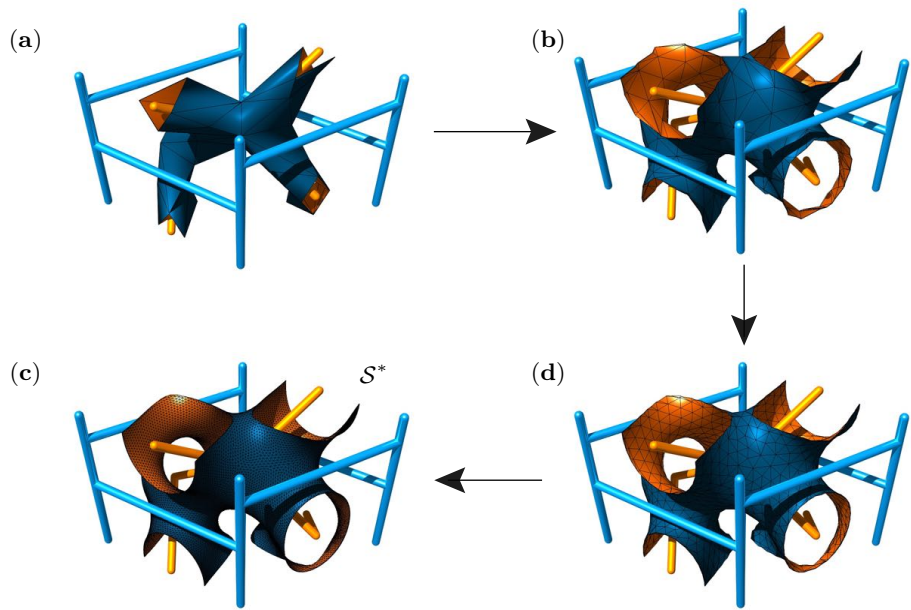
\begin{figure*}[h]
	\labellist
	\small\hair 2pt
	\pinlabel $\bf{(c)}$ at 10 200
	\pinlabel $\bf{(d)}$ at 400 200
	\pinlabel $\bf{(b)}$ at 400 450
	\pinlabel $\bf{(a)}$ at 10 450
	\pinlabel $\mathcal{S}^*$ at 270 200
	\endlabellist
	\centering
	\includegraphics[scale=0.5]{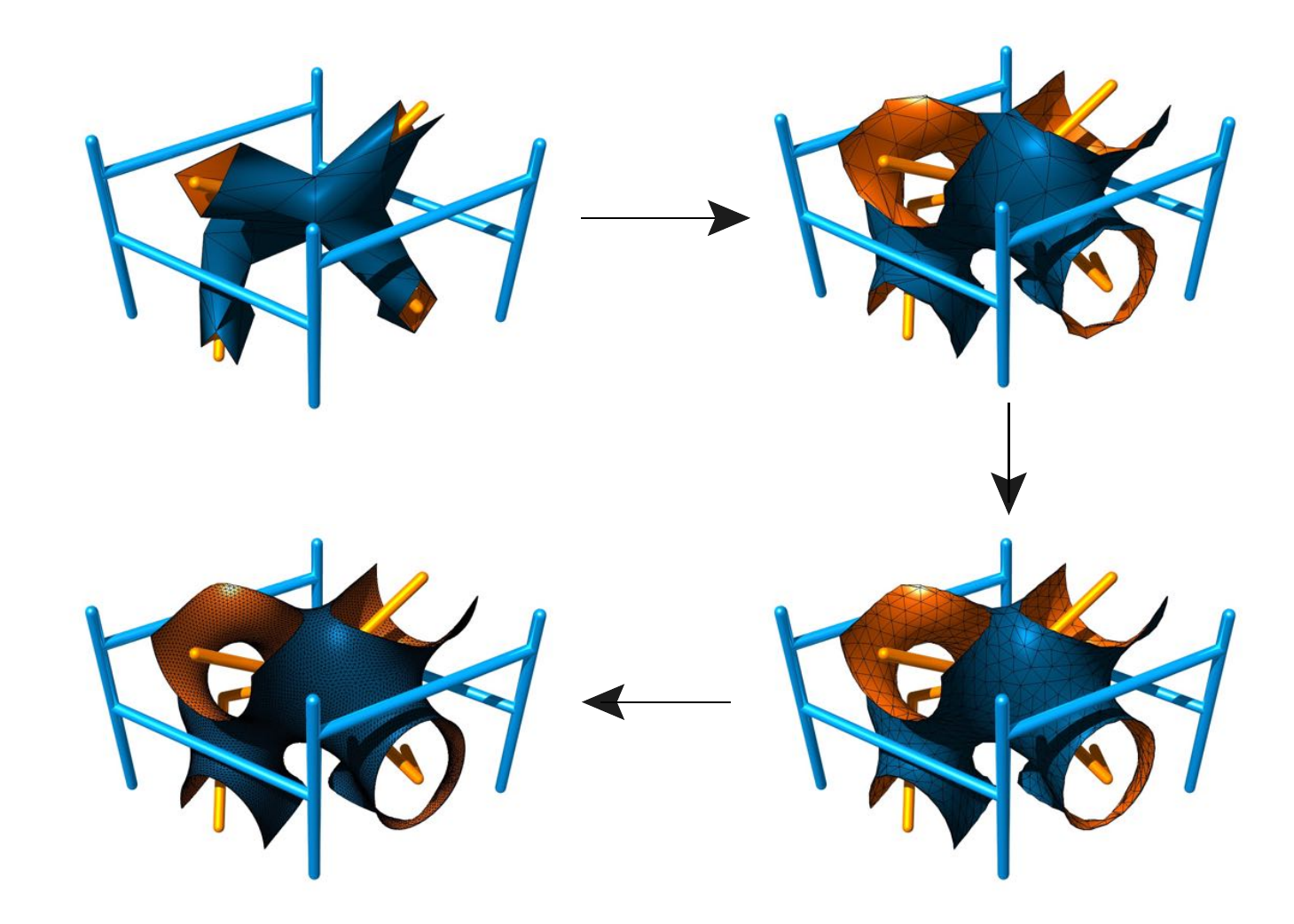}
	\caption{Four different stages illustrating the minimisation of the area functional using conjugate gradient descent method in Surface Evolver \cite{Brakke_1992}. See Appendix A2.2. The boundary conditions respect the translational periodicities imposed by the $P6_222$ space group symmetry with lattice parameters $c= a=1$. The blue coloured network is the {\texttt{qzd}}  network and the orange coloured network is the {\texttt{qtz}} (quartz) network. Figure (a) shows a triangulated tubular neighbourhood of the quartz network and figure (c) shows the final surface, $\mathcal{S}^*$ after convergence.}
	\label{fig:test1}
\end{figure*}

\vfill
\newpage

\section{\label{supp-mat-asympt-behav} Asymptotic behavior}

We now prove that
\begin{lemma}
	\[
		\lim_{\im\tau \to +\infty} \arg \psi(\tau)
		= \lim_{\im\tau \to +\infty} \arg \int_0^{\tau/3} G(z) dz
		= \frac{1}{2} - \frac{\re\tau}{3}.
	\]
\end{lemma}

\begin{proof}
	In the limit $\im\tau \to +\infty$, by the expansion~\cite[(20.2.1)]{DLMF}
	\[
		\theta(z;\tau) = 2 q^{1/4} \sin(z) + o(q^{1/4}),
	\]
	where $q = \exp(i \pi \tau)$, we have
	\begin{align*}
		G(z)
		&\sim
		\sqrt{-\frac{\sin(\pi 3z)}{\sin(\pi (3z - \tau)} e^{i \pi (2z - \tau/3)}}\\
			&= \sqrt{-\frac{e^{i \pi 3z} - e^{-i \pi 3z}}{e^{i \pi (3z - \tau)} - e^{-i \pi (3z - \tau)}}} e^{i \pi (z - \tau/6)},
		\end{align*}
		as $\im\tau \to +\infty$.

		Now we divide $\psi(\tau)$ into two integrals
		\[
			\psi(\tau) = \psi_1(\tau) + \psi_2(\tau) = \int_0^{\tau/6} G(z) dz + \int_{\tau/6}^{\tau/3} G(z) dz
		\]

		In the integral $\psi_1(\tau)$, $\tau - 3z > \tau/2$.  So as $\im \tau \to +\infty$, we 
		\begin{align}
			\psi_1(\tau)
			&\sim \int_0^{\tau/6} \sqrt{\frac{e^{i \pi 3z} - e^{-i \pi 3z}}{e^{-i \pi (3z - \tau)}}} e^{i \pi (z - \tau/6)} dz\nonumber\\
			&= e^{i \pi \tau/3} \int_0^{\tau/6} \sqrt{e^{i \pi 6z} - 1} e^{i \pi z} dz\nonumber\\
			&= \frac{i}{\pi} e^{i \pi \tau/3} \int_1^{e^{-i\pi\tau/6}} \frac{\sqrt{w^6-1}}{w^2} dw  & (w = e^{-i\pi z}).\label{eq:app1}
		\end{align}
		Note that
		\[
			\Big| \frac{\sqrt{w^6-1}}{w^2} \Big| < |w|,
		\]
		we have
		\begin{align*}
			| \text{\eqref{eq:app1}} |
			&\le \frac{1}{\pi} e^{-\pi \im\tau/3} \int_1^{e^{\pi\im\tau/6}} x dx \\
			& = \frac{1}{2\pi} (1 - e^{-\pi\im\tau/3}) \to \frac{1}{2\pi}
		\end{align*}
		as $\im\tau \to +\infty$.  So $\psi_1(\tau)$ is bounded.

		In the integral $\psi_2(\tau)$, $3z > \tau/2$.  So as $\im \tau \to +\infty$, we 
		\begin{align*}
			\psi_2(\tau)
			&\sim \int_{\tau/6}^{\tau/3} \sqrt{\frac{e^{-i \pi 3z}}{e^{i \pi (3z - \tau)} - e^{-i \pi (3z - \tau)}}} e^{i \pi (z - \tau/6)} dz\\
			&= \int_{0}^{\tau/6} \sqrt{\frac{e^{-i \pi (\tau - 3\zeta)}}{e^{- i \pi 3\zeta} - e^{i \pi 3\zeta}}} e^{i \pi (\tau/6 - \zeta)} dz & (3\zeta = \tau - 3z)\\
			&= e^{-i \pi \tau/3} \int_0^{\tau/6} \frac{e^{-i\pi \zeta}}{\sqrt{e^{-i\pi 6\zeta} -1}} d\zeta\\
			&= \frac{i}{\pi} e^{-i \pi \tau/3} \int_1^{e^{-i\pi \tau/6}} \frac{du}{\sqrt{u^6 -1}} & (u = e^{-i \pi \zeta})\\
			&= \frac{i}{\pi} e^{-i \pi \tau/3} \int_1^{e^{i\pi \tau/6}} \frac{-v dv}{\sqrt{1 - v^6}} & (v = 1/w)
		\end{align*}
		Note that the integral in the last line converge to the bounded real integral
		\[
			I = \int_0^1 \frac{x dx}{\sqrt{1 - x^6}} \in \mathbb{R}
		\]
		as $\im\tau \to +\infty$, so
		\[
			\psi_2(\tau) \sim \frac{i}{\pi} I e^{-i\pi \tau/3}.
		\]

		As $\im \tau \to +\infty$, $\psi_2$ is dominant, so
		\[
			\lim_{\im\tau \to +\infty} \arg \psi(\tau) =
			\lim_{\im\tau \to +\infty} \arg \psi_2(\tau) = \frac{1}{2} - \frac{\tau}{3}.
		\]
	\end{proof}

	\section{Proof of Lemma \ref{lemma:screw-symmetry}}\label{sec:proof-of-lemma-screw}
	\begin{proof}
		The height differential $dh$ is invariant under $S$, hence descends
		holomorphically to the quotient $(M/\Lambda)/S$.  Since there is no
		holomorphic differential on the sphere, the genus of $(M/\Lambda)/S$ cannot
		be $0$.

		Recall the Riemann--Hurwitz formula
		\[
			g = n(g' - 1) + 1 + B/2.
		\]
		In our case, $g=4$ is the genus of $M/\Lambda$, $g'$ is the genus of
		$(M/\Lambda)/S$, $n$ is the degree of the quotient map, and $B$ is the total
		branching number.  Since the $S$ is of order at least $2$, we conclude
		immediately that $g' < 3$.  It remains to examine the case $g'=2$.  Without
		loss of generality, we may assume the screw axis to be vertical.  When
		$g'=2$, we have $3 = n + B/2$.  Recall that $n>1$ and that the Gauss map $G$
    represents $M/\Lambda$ as a $3$-sheeted branched cover of $\mathbb{S}^2$
    with branching number $12$.  More specifically, we have $12$ branch points
    of branching number $1$, none fixed by the screw symmetry.

		One solution is given by $n=3$ and $B=0$.  Then $S$ is a screw symmetry of
		order $3$ and $G^3$ represents $(M/\Lambda)/S$ as a $3$-sheeted branched
		cover of $\mathbb{S}^2$ with branching number $4$.  By Riemann--Hurwitz
		formula, $(M/\Lambda)/S$ has genus
    \[
    	g' = 3 \times (0 - 1) + 1 + 4/2 = 0 \ne 2,
    \]
    which is impossible.

		Another solution is given by $n=2$ and $B=2$.  Then $S$ is a screw symmetry
		of order $2$ and $G^2$ represents $(M/\Lambda)/S$ as a $3$-sheeted
		branched cover of $\mathbb{S}^2$ with branching number $6$.  Again By
		Riemann--Hurwitz formula, $(M/\Lambda)$ has genus
    \[
    	g' = 3 \times (0 - 1) + 1 + 6/2 = 1 \ne 2,
    \]
    again impossible.

		So $g'=1$ is the only possibility.
	\end{proof}

  \section{Proof of Lemma \ref{lem:branchH}\label{proof-of-lem:branchH}}
	\begin{proof}
		The screw symmetry of order $6$ descends to the quotient torus as a
		translation that permutes the zeroes and, respectively, the poles.  Let $t \in
		\C$ be the translation vector.  Then $t$ must be a 3-division point.  That
		is, $3t$ must be a lattice point.

		We may assume that the branch point at $0$ corresponds to a zero of $G^2$.
		Then $t$ and $-t$ are also zeroes of $G^2$.  Let the poles of $G^2$ be $p$,
		$p+t$ and $p-t$.  Then Abel's Theorem requires that $3p$ is a lattice point.
		In other words, $p$ is a 3-division point.

		We have then proved that all branch points are at 3-division points.
	\end{proof}

	\section{Proof of Proposition \ref{existence-proof-proposition} (existence)}\label{suppmatsec:proof-of-existence}
	\begin{proof}
		We examine the angles $\theta_v$ and $\theta_h$ on the boundaries of
		$\Omega_t$.

		On the vertical line $\re\tau = 0$, we see immediately that $0 < \theta_v <
		\pi/2$.  With $\theta = \pi/2$, this line corresponds to the H'-T family, we
		know very well that the image of $\tau$ is directly above the image of $0$
		when $dh = e^{-i \pi/2} dz$.  Hence $\theta_h = \pi/2 > \theta_v$.  In
		particular, the image of $\int G \cdot dz$ with $\theta = \pi/2$ is shown in
		Figure~\ref{fig:HTflat} (left).

		On the vertical line $\re\tau = 1$, we see immediately that $\theta_v = 0$.
		With $\theta = 0$, this line corresponds to the conjugate of H'-T surfaces.
		In particular, the image of $\int G \cdot dz$ with $\theta = 0$ is shown in
		Figure~\ref{fig:HTflat} (right), from which we can read that $\theta_h > 0 =
		\theta_v$.

		\begin{figure*}
			\includegraphics[width=0.3\textwidth]{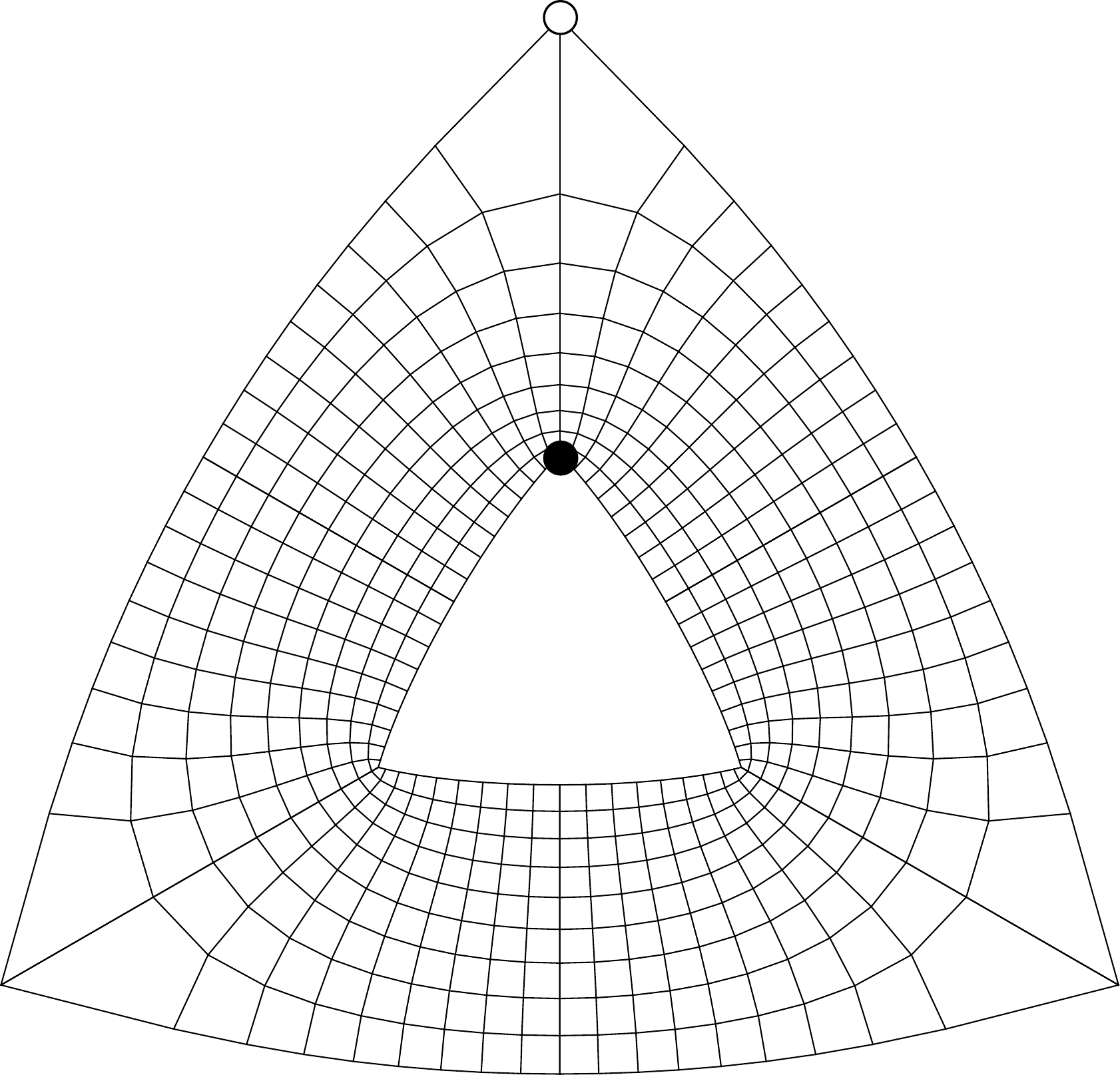}
			\includegraphics[width=0.3\textwidth]{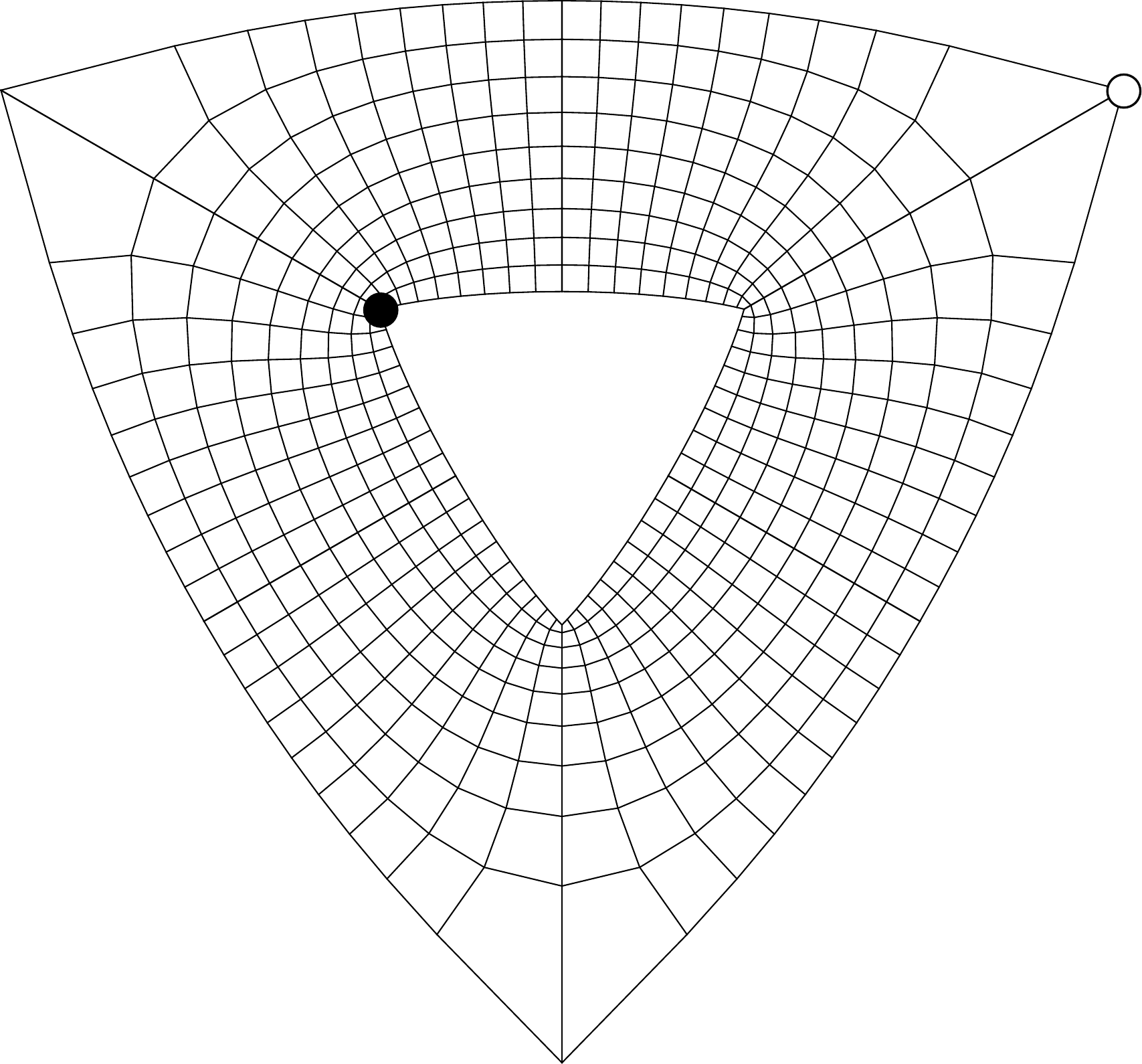}
			\caption{
				Image of $\int G \cdot dz$ when $\re\tau = 0$ (left) and when $\re\tau =
				1$.  The solid circle is the image of $0$, the empty circle is the image
				of $\tau/3$.
			}
			\label{fig:HTflat}
		\end{figure*}

		As $\im\tau \to \infty$, we see immediately that $\theta_v \to 0^+$.  In
		Appendix~\ref{supp-mat-asympt-behav}, we have compute that $\theta_h
		\to (1/2-\re\tau/3)\pi$.  Hence for any value of $\re\tau < 1$, we have
		$\theta_h > \theta_v$ as long as $\im\tau$ is sufficiently large.

		When $\tau$ is on the circular arc $|\tau-1/3| = 1/3$, the Gauss map
		exhibits many symmetries similar to CLP surfaces; see Figure~\ref{fig:arcs}
		(left).  In particular, $G$ is real positive on the straight segment from
		$0$ to $\tau/3$, which gives immediately that $\theta_h = \arg \psi(\tau) =
		\arg \tau$.  But it follows from elementary geometry that $\theta_v >
		\arg\tau = \theta_h$.

		\begin{figure*}
			\includegraphics[width=0.4\textwidth]{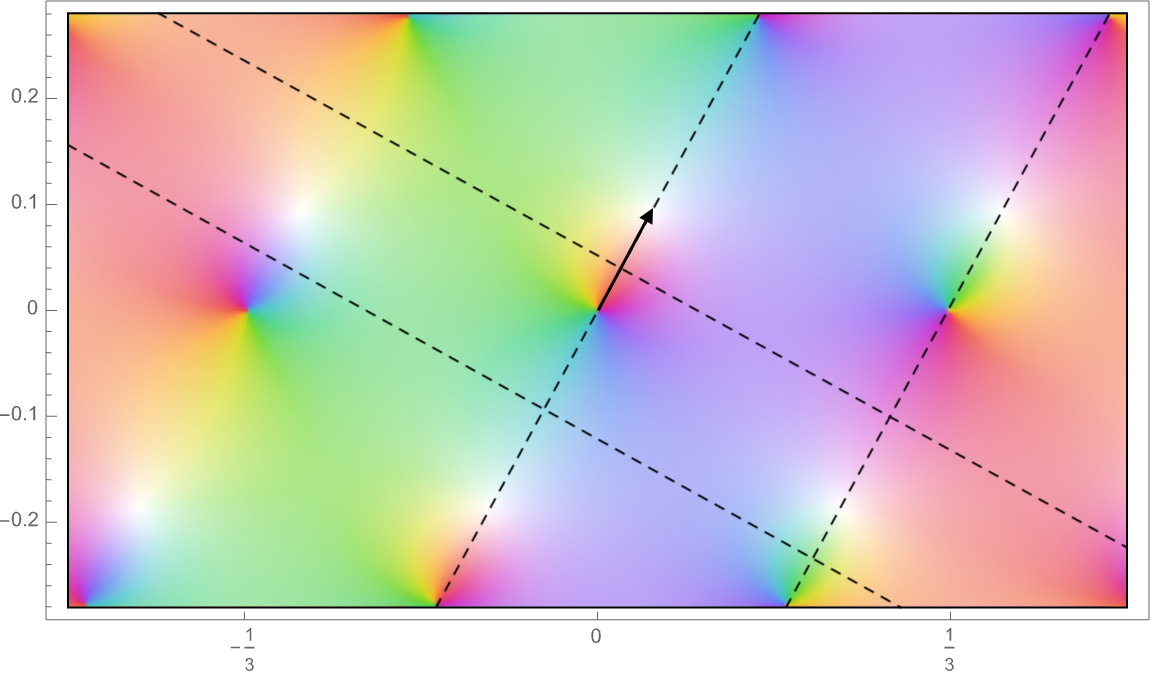}
			\includegraphics[width=0.4\textwidth]{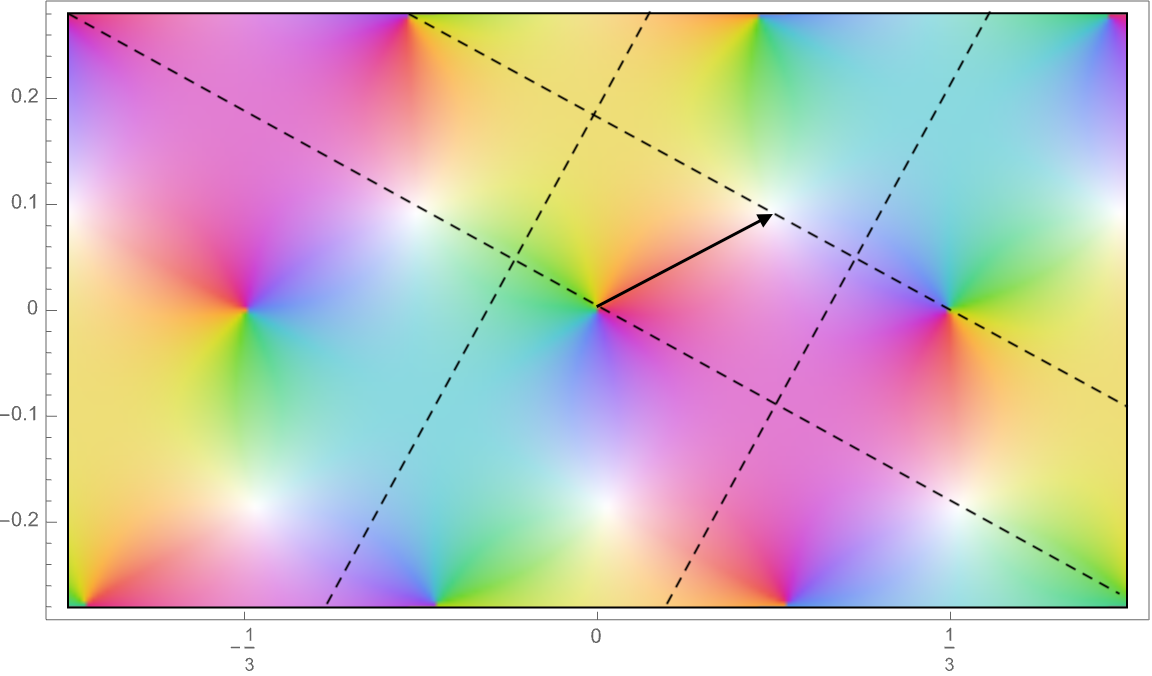}
			\caption{
				Plot of $G^2$ for with $|\tau - 1/3| = 1/3$ (left) and $|\tau-2/3| =
				1/3$.  The arrow points from $0$ to $\tau/3$.  The dashed lines
				highlight the symmetries.
			}
			\label{fig:arcs}
		\end{figure*}

		When $\tau$ is on the circular arc $|\tau-2/3| = 1/3$, the Gauss map again
		exhibits many symmetries, similar to the previous case; see
		Figure~\ref{fig:arcs} (right).  In particular
		\begin{itemize}
			\item on the straight segment from $0$ to $(1-\tau)/3$, $\arg(G) = -\pi/6$ and $|G| < 1$;
			\item on the straight segment from $(1-\tau)/3$ to $(1+\tau)/6$, $|G|=1$ and $\arg(G)$ decrease from $-\pi/6$ to $-\pi/3$;
			\item on the straight segment from $(1+\tau)/6$ to $\tau/3$, $\arg(G) = -\pi/3$ and $|G| > 1$;
			\item on the straight segment from $\tau/3$ to $(2\tau - 1)/3$, $\arg(G) = \pi/6$ and $|G| > 1$;
			\item on the straight segment from $(2\tau - 1)/3$ to $-(1-\tau)/6$, $|G|=1$ and $\arg(G)$ increase from $\pi/6$ to $\pi/3$;
			\item on the straight segment from $-(1-\tau)/6$ to $0$, $\arg(G) = \pi/3$ and $|G| < 1$.
		\end{itemize}
		The region bounded by these segments is exactly the region bounded by the
		dashed lines in Figure~\ref{fig:arcs} (right).

		This allows us to explicitly sketch the image of the Weierstrass
		parametrization with $\theta = \arg(1-\tau) = \theta_v - \pi/2$.  Up to
		translation and scaling, it is bounded by six straight segments between the
		points
		\[
			(0, 2, h),\quad (-\sqrt{3}, 1, h),\quad (0, -2, h),\quad (0, -2, 0),\quad (-\sqrt{3}, -1, 0),\quad (0, 2, 0),
		\]
		in this cyclic order.  See Figure~\ref{fig:CLP} (left).  Then by
		conjugation, we can sketch the image with $\theta = \arg(\tau/2 - 1/6) =
		\theta_v$.  It is a minimal surface with free boundary condition on the
		faces of a prism over a kite; see Figure~\ref{fig:CLP} (right).  We see that
		the $y$-coordinates of the images of $0$ and $\tau/3$ switched order,
		implying that $\theta_h \in (\theta_v - \pi/2 , \theta_v)$.  In particular,
		$\theta_h < \theta_v$.

		\begin{figure*}
			\includegraphics[width=0.2\textwidth]{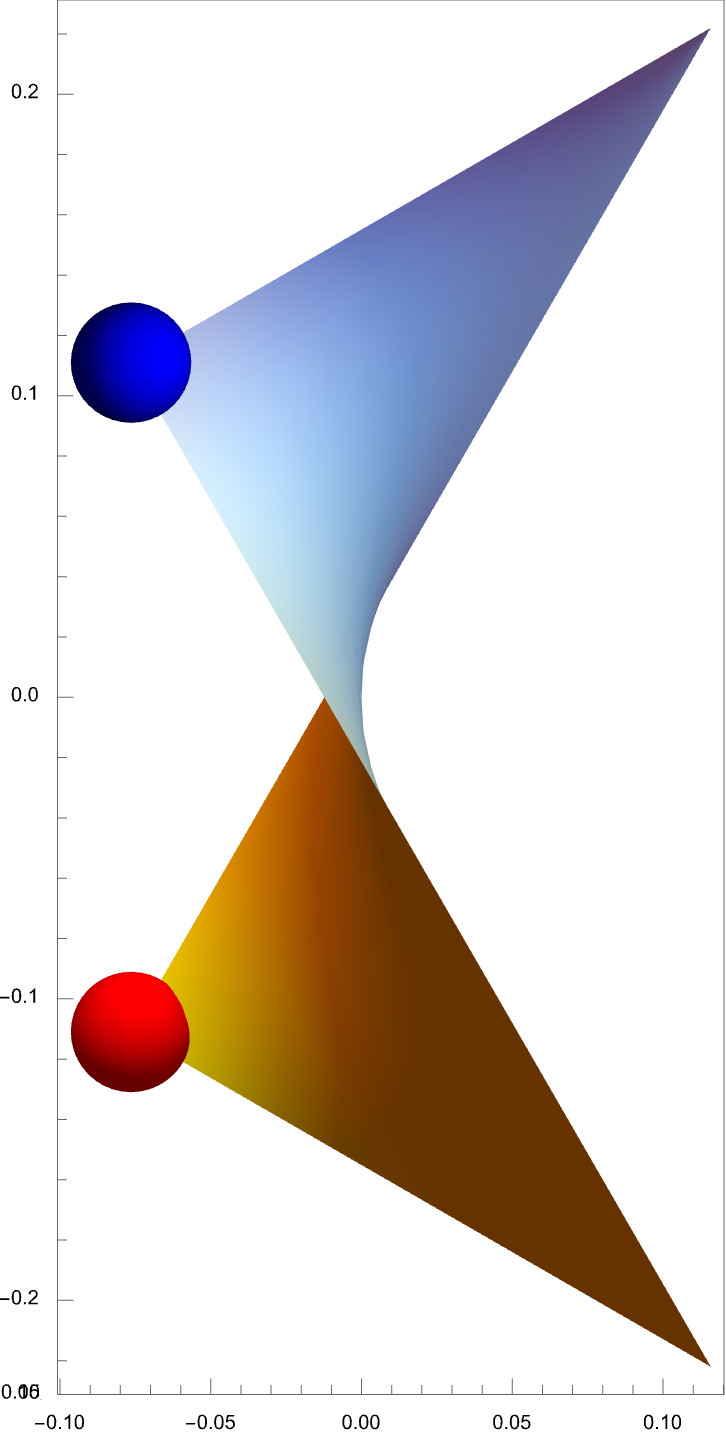}
			\includegraphics[width=0.2\textwidth]{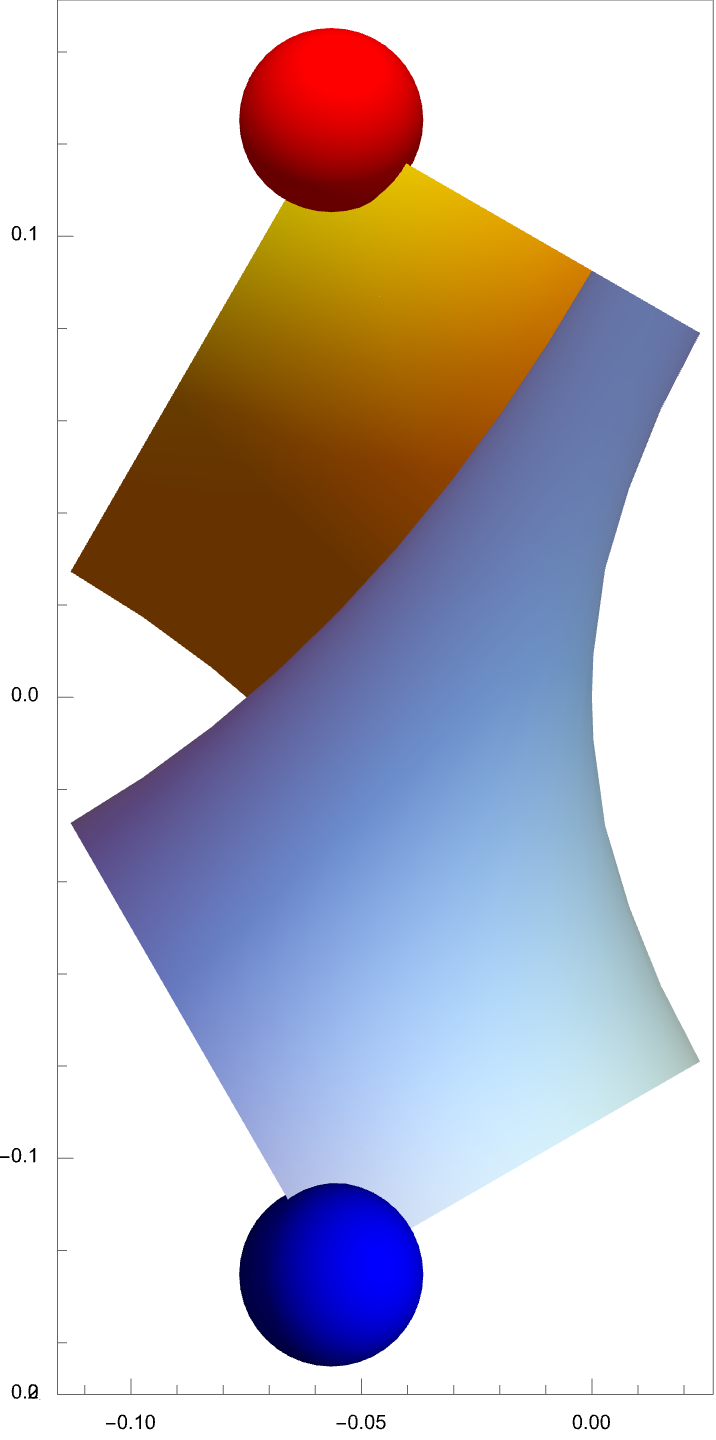}
			\caption{
				Image of Weierstrass parametrization assuming $\theta = \theta_v -
				\pi/2$ (left) and $\theta = \theta_v$ (right).  Both seen from above.
				The red point is the image of $0$.  The blue point is the image of
				$\tau/3$.
			}
			\label{fig:CLP}
		\end{figure*}

		Note that $\theta_h$ and $\theta_v$ are both real analytic functions in the
		real and imaginary part of $\tau$, hence the solution set of the period
		condition~\eqref{eq:period} is an analytic set.  By the continuity, we
		conclude that the solution set contains a connected component that separates
		the circular arcs from the vertical lines and the infinity.  Because of the
		analyticity, we may extract a continuous curve from the connected component.
		In particular, this curve must tend to $0$ in one limit (correspond to
		Karcher--Scherk saddle towers with six wings), and tend to $1$ in the other
		limit (correspond to doubly periodic Scherk surfaces).

		Surfaces with the same symmetry near the saddle tower limit have been
		constructed in~\cite{chen2024} by gluing constructions.  They are proved to
		be embedded and unique, hence must belong to the curve above near the limit
		0.  Their embeddedness then ensures the embeddedness of all TPMSs on the
		curve.  This follows from~\cite[Proposition~5.6]{weyhaupt2006}, which was
		essentially proved in~\cite{meeks1990}.
	\end{proof}

	\end{document}